\documentclass{amsart}

\usepackage{amsfonts,amsmath,amssymb,amsthm,stmaryrd}
\usepackage{geometry}
\usepackage{graphicx,xcolor}
\usepackage[hidelinks]{hyperref}
\usepackage{tikz} \usetikzlibrary{cd}
\usepackage[all]{xy}
\usepackage{dutchcal}
\usepackage{bm}

\usepackage[backend=biber,style=alphabetic,sorting=nyt,giveninits=true,isbn=false,doi=false,url=false,maxalphanames=5,minalphanames=4,maxnames=5,minnames=4]{biblatex}

\usepackage{zref-clever} \newcommand\Cref[1]{\zcref[S]{#1}} 

\newcommand\newtheoremx[2]{%
\AddToHook{env/#1/begin}{
\zcsetup{countertype={theorem=#1}}}
\newtheorem{#1}[theorem]{#2}
}

\newtheorem{theorem}{Theorem}[section]
\newtheoremx{claim}{Claim}
\newtheoremx{proposition}{Proposition}
\newtheoremx{corollary}{Corollary}
\newtheoremx{lemma}{Lemma}
\newtheoremx{conjecture}{Conjecture}

\newtheorem{introtheorem}{Theorem}

\zcsetup{countertype={introtheorem=theorem}}

\zcRefTypeSetup{conjecture}{
  name-sg = {conjecture},
  Name-sg = {Conjecture},
  name-pl = {conjectures},
  Name-pl = {Conjectures}
}

\theoremstyle{definition}
\newtheoremx{definition}{Definition}
\newtheoremx{notation}{Notation}
\newtheoremx{assumption}{Assumption}

\theoremstyle{remark}
\newtheoremx{example}{Example}
\newtheoremx{remark}{Remark}
\newtheoremx{question}{Question}

\numberwithin{equation}{section}

\newcommand\id{\operatorname{id}}

\newcommand{\Rep}{\operatorname{Rep}}
\newcommand{\Res}{\operatorname{Res}}
\newcommand{\Ind}{\operatorname{Ind}}
\newcommand{\Hom}{\operatorname{Hom}}
\newcommand{\End}{\operatorname{End}}

\newcommand{\Fr}{\operatorname{Fr}}

\newcommand{\supp}{\operatorname{supp}}

\newcommand{\FPdim}{\operatorname{FPdim}}

\newcommand{\Ver}{\operatorname{Ver}}
\newcommand{\Tilt}{\operatorname{Tilt}}

\newcommand{\sVec}{\mathrm{sVec}}
\newcommand{\op}{\mathrm{op}}

\newcommand{\blambda}{\bm{\lambda}}
\newcommand{\mN}{\mathbb{N}}
\newcommand{\mZ}{\mathbb{Z}}
\newcommand{\mA}{\mathbb{A}}
\newcommand{\mF}{\mathbb{F}}

\newcommand{\mG}{\mathbb{G}}
\newcommand{\mP}{\mathbb{P}}

\newcommand{\cA}{\mathcal{A}}

\newcommand{\cC}{\mathcal{C}}
\newcommand{\cV}{\mathcal{V}}
\newcommand{\cD}{\mathcal{D}}
\newcommand{\cP}{\mathcal{P}}
\newcommand{\cT}{\mathcal{T}}
\newcommand{\cF}{\mathcal{F}}

\newcommand{\cK}{\mathcal{K}}

\newcommand{\bk}{\Bbbk}
\newcommand{\bV}{\mathbf{V}}
\newcommand{\bO}{\mathbf{O}}

\newcommand{\unit}{\mathbf{1}}
\newcommand\kk{\Bbbk}
\renewcommand\o{\otimes}
\newcommand{\tto}{\twoheadrightarrow}

\newcommand{\ulambda}{{\underline{\lambda}}}

\newcommand{\ex}{\mathrm{ex}}

\newcommand\gr{\operatorname{gr}}
\newcommand\im{\operatorname{im}}
\newcommand\lmod{\text{-}\mathrm{mod}}
\newcommand\mfrak{\mathfrak m}
\renewcommand\o{\otimes}
\renewcommand\:\colon

\newcommand\ol{\overline}
\newcommand\Pol{\mathcal{Pol}}
\newcommand\Vect{\mathrm{Vect}}
\newcommand\SL{\mathrm{SL}}
\newcommand\len{\operatorname{len}}
\newcommand\tforall{\text{for all }}

\newcommand\sgd{\operatorname{sgd}}
\newcommand\Sym{\operatorname{Sym}}

\title[The Benson--Etingof--Ostrik conjecture for finite tensor categories]{The Benson--Etingof--Ostrik conjecture\\for finite tensor categories}
\author{Kevin Coulembier}
\address{School of Mathematics and Statistics, University of Sydney, Australia}
\email{kevin.coulembier@sydney.edu.au}
\author{Johannes Flake}
\address{Mathematical Institute, University of Bonn, Germany}
\email{flake@math.uni-bonn.de}

\begin{document}

\begin{abstract}
    We prove that every finite symmetric tensor category over an algebraically closed field of positive characteristic admits a symmetric tensor functor to a higher Verlinde category. This was conjectured for arbitrary symmetric tensor categories of moderate growth by Benson, Etingof and Ostrik \cite{BEO}. The proof uses higher Frobenius functors, introduced by the authors in \cite{CF1}. The main effort is directed towards a proof of a conjecture from \cite{CF1} about thick tensor ideals in the representation categories of elementary abelian $p$-groups, which implies that the higher Frobenius functors are symmetric monoidal. In fact, we prove a refinement of the latter conjecture that relates tensoring with restricted Steinberg modules to an iterated generalised cohomology of modules of elementary abelian $p$-groups. 
\end{abstract}

\maketitle


\section{Introduction}

Let $\bk$ be an algebraically closed field.
We consider $\bk$-linear abelian symmetric rigid monoidal categories satisfying some finiteness conditions, which are known as (symmetric) \emph{tensor categories}, see \cite{EGNO}.

If the characteristic of $\bk$ is zero, then Deligne proved in \cite{Del02} that every tensor category of \emph{moderate growth} (which means that the length of tensor powers of a given object grows at most exponentially) is a representation category of an affine supergroup scheme. By the principle of tannakian reconstruction \cite{Del90}, this is equivalent to saying that every tensor category of moderate growth admits a tensor functor to (i.e., \emph{fibres} over) the tensor category of super vector spaces. Yet equivalently, Deligne's theorem states that the category of super vector spaces and its subcategory of vector spaces are the only \emph{incompressible} tensor categories of moderate growth, see \cite{CEO2}.

In this paper, we consider tensor categories (of moderate growth) over an algebraically closed field $\bk$ of characteristic~$p>0$. A naive version of Deligne's theorem is known to be false. Indeed, by \cites{BE, BEO,Cou-abelian}, there is an infinite sequence of incompressible tensor categories called \emph{higher Verlinde categories}
$$
\Ver_p \subsetneq \Ver_{p^2} \subsetneq \dots \subsetneq \Ver_{p^\infty} = \bigcup_{r\ge1} \Ver_{p^r} ,
$$
where $(\Ver_{p^r})_{1\le r<\infty}$ are defined as the abelian envelopes of the quotients of $\Tilt \SL_2$.
The main conjecture of the field asserts that $\Ver_{p^\infty}$ plays the role of super vector spaces in characteristic $0$:

\begin{conjecture}[{\cite[Conjecture~1.4]{BEO}}] \label{conj:BEO} Any symmetric tensor category of moderate growth over an algebraically closed field of positive characteristic $p$ fibres over $\Ver_{p^\infty}$.
\end{conjecture}

The most important positive result yet towards \Cref{conj:BEO} is a characterisation in \cite{CEO} of those tensor categories that fibre over $\Ver_p$, which built on the earlier special cases of fusion categories \cite{Os} and finite tensor categories \cite{EOf}. It is shown in \cite{CEO} that a tensor category of moderate growth admits a tensor functor to $\Ver_p$ if and only if a certain Frobenius functor, originally defined by Ostrik in \cite{Os}, acts exactly. We say the Frobenius functor \emph{detects} fibring over $\Ver_p$. Crucial in the definition of Ostrik's functor was a realisation of $\Ver_p$ as the semisimplification of the category of representations of the cyclic $p$-group $C_p$, rather than of $\Tilt SL_2$.

In an attempt at generalising these ideas, the authors in \cite{CF1} constructed realisations of every higher Verlinde category $\Ver_{p^r}$ in terms of categories of representations of the elementary abelian $p$-group $C_p^r$. We then constructed candidates for generalisations of Ostrik's functor to higher Frobenius functors for all $r$. However, some properties (collected using the notion \emph{$\bO$-functor}) of our higher Frobenius functors remained conjectural. In fact, they were shown to rely on the following conjecture on representations of elementary abelian $p$-groups.

\begin{conjecture}[{\cite[Conjecture~5.4.2(1)]{CF1}, \cite{CF2}}] \label{conj:CF} For any $r\ge1$, fix an embedding of $C_p^r$ into $\SL_2(\kk)$. Let $S\in\Rep C_p^r$ be the restriction of the tilting module with highest weight $p^{r-1}-1$. Then the functor $S\o-$ on $\Rep C_p^r$ has values in the subcategory of restrictions of $\SL_2$-tilting modules.
\end{conjecture}

In this paper, we resolve \Cref{conj:CF} by proving a refined description of the functor $S\o-$. Using an iterated cohomology $H_X$ (depending on an embedding of $C_p^r$ into $\SL_2(\bk)$) for modules $X$ of $C_p^r$ that we introduce in \Cref{sec:St-Cpr}, and denoting restricted $\SL_2$-simples by the symbols $L(i)$ according to their highest weight $i\in\mN$, we show:

\begin{introtheorem}[\Cref{thm:wet-spaghetti}] \Cref{conj:CF} is true; more precisely, for $X\in\Rep C_p^r$ and $S$ as above: 
\begin{equation} \label{eq:wet-spaghetti}
S\o X\cong\bigoplus_{0\le i< p^r} (S\o L(i))^{\oplus \dim H_X(i)}.    
\end{equation}
As a consequence, the functors $\Phi^V_\cC$ from \cite{CF1} are $\bO$-functors.
\end{introtheorem}

We then use these $\bO$-functors to show a finite version of \Cref{conj:BEO}.

\begin{introtheorem}[\Cref{thm:BEO}] \label{thm:B} \Cref{conj:BEO} is true for \textbf{finite} symmetric tensor categories; in fact, any such category fibres over $\Ver_{p^r}$ for some $1\le r<\infty$.
\end{introtheorem}
Moreover, by \Cref{thm:general}, see \Cref{rem:toomany} for a more precise discussion, the functors $\Phi^V_\cC$ detect fibring over $\Ver_{p^r}$ for finite symmetric tensor categories.

Around the time we finished writing this article, Alexander Sherman independently found an alternative proof \cite{Sherman} of \Cref{conj:CF}, also using AI. While his proof does not imply \Cref{eq:wet-spaghetti}, it is shorter and does seem conceptually illuminating to us.

\subsection*{Implications and new directions}

We refer to \Cref{sec:cons} for some immediate applications of the main results in the paper. Here we briefly discuss the broader potential implications of our results on current and future research.

\subsubsection*{Tensor categories}

Two emerging lines of research in (finite) tensor categories gain considerably in importance due to \Cref{thm:B}. These are the study of polynomial functors in higher Verlinde categories, in the sense of \cite{CoNew}, and the study of finite (or more specifically infinitesimal) group schemes in higher Verlinde categories. Indeed, the ``classification'' of the latter is now equivalent to the ``classification'' of all finite tensor categories. Moreover, via \Cref{thm:B}, both topics together should play a major role in resolving \cite[Conjecture~2.18]{EO-finite} about finite generation of cohomology, in the symmetric case.

\Cref{thm:B} constitutes significant progress towards \Cref{conj:BEO}. Indeed, the latter now essentially reduces to the conjecture that any finitely generated tensor category of moderate growth admits a tensor functor to a finite tensor category.  Nonetheless, it seems that fundamentally new ideas and techniques will be required to prove the full \Cref{conj:BEO}. Perhaps the most significant implication of \Cref{thm:B} in this regard is the convincing evidence it provides for \Cref{conj:BEO}. Indeed, when the authors of \cite{BEO} formulated this conjecture they described it as ``\emph{a (perhaps
imprudently) bold conjecture}''. We feel this description is now no longer applicable.

\subsubsection*{Modular representation theory}
Beyond tensor categories, we expect that our results will open new directions in modular representation theory. Even the validity of \Cref{conj:CF} seemed surprising to experts in the field. Its resolution via the explicit formula in \Cref{eq:wet-spaghetti} should be a helpful new technique. It can also be seen as a first step towards a theory of (additive) tensor ideals in categories of modular representations, on which almost nothing seems to be known. The iterated cohomology $H_X$ we introduce to formulate \Cref{eq:wet-spaghetti} seems to be a new and interesting invariant of a module. As it can predict the tensor product with $S$, and in particular, carries information on the annihilation of extensions under the functor $S\o-$, it could be interpreted as a refinement of classical support theory. Note that $S$ and hence all non-projective indecomposable summands in tensor products involving $S$ have a classical support of one point in projective space, but classical support theory does not give any more insights into these tensor products. Such insights are provided by $H_X$, and in fact, we have a definition of $H_X$ and a tensor product decomposition result for every point in a dense subset of projective space. 

Another interpretation of the now proven \Cref{conj:CF} is that any representation of an elementary abelian $p$-group is ``just a tensor product with $S$ away from'' an $\SL_2$-representation. This suggest new ways of using algebraic group representation theory to study representations of finite groups. It would be interesting to study if such a connection can be widened to include larger algebraic groups. 

Note that while \Cref{thm:B} is specific to elementary abelian $p$-groups, we expect it to be of relevance for modular representations of general finite groups, due to the classical insight in modular representation theory that the restrictions to elementary abelian subgroups control many aspects of the theory, such as projectivity, support, or nilpotence in cohomology.

\subsection*{Use of AI} The main ideas in the proofs in \Cref{sec:St} and \Cref{sec:BEO} were originally obtained using {GPT-6} Astra, using as an input the formula in \Cref{eq:wet-spaghetti}, which was obtained previously using extensive numerical experiments with the help of Claude~{4.6--5} Opus. This document was then written by hand and checked with {GPT-6} Astra.

\subsection*{Remark} Four days after the first version of this paper was submitted to arXiv, a commercial AI company released a proof written by AI, of the same two conjectures that are resolved in this paper. More concretely, it announces a proof of Theorem B and a very different proof of \Cref{conj:CF}, which does not prove the concrete \Cref{eq:wet-spaghetti}.

\subsection*{Acknowledgments} We thank an anonymous referee of \cite{CF1} whose questions essentially about a formula like \Cref{eq:wet-spaghetti} triggered us to look into this again. 
J.F.~ was supported by the Hausdorff Center for Mathematics Bonn. J.F.'s Claude license was provided by the 2024 Max Planck--Humboldt Research Award (Max Planck Society / Alexander von Humboldt Foundation) bestowed on Geordie William\-son and hosted by Catharina Stroppel. K.C. was supported by ARC grants FT220100125 and DP250100762. We thank Alex Sherman for sharing his independent proof.


\section{Preliminaries} 

\subsubsection{Notation}

Throughout the paper, $\kk$ is an algebraically closed field of characteristic $p>0$. All vector spaces are over $\kk$, and tensor products are over $\kk$ if not indicated otherwise. We write $\Vect$ for the category of finite-dimensional vector spaces. The symmetric group $S_n$ is the group of permutations of $\{1,2,\ldots, n\}$. For a finite group $G$, we let $\Rep G$ denote its category of finite-dimensional representations over $\bk$. We will use the same notation for an affine group scheme $G$ and its category of finite-dimensional rational representations.

\subsection{Elementary abelian groups}

In this paper, the central role will be played by elementary abelian $p$-groups. Here we review how we will view them as subgroups of $\SL_2(\bk)$ and of symmetric groups, and the basics of cohomological support theory. For an elementary abelian $p$-group $C_p^r$, for $r\in\mZ_{>0}$, we denote the generator of the $j$th copy of $C_p$ by $g_j$.

\subsubsection{Subgroups of \texorpdfstring{$\SL_2$}{SL2}}\label{sec:inSL2}

We will use embeddings of $C_p^r$ into the unipotent radical of the Borel subgroup of lower-triangular matrices of $\SL_2$ (this radical is thus just a copy of $\mG_a$). In other words, we consider homomorphisms
\[\iota_{\ulambda}:C_p^r\;\to\;(\bk,+)\;\to\; SL_2(\bk),\quad g_i\mapsto \left(\begin{array}{cc}
    1 &0  \\
    \lambda_i & 1 
\end{array}\right),\]
for non-zero $(\lambda_1,\ldots,\lambda_r)\in\mA^r=\mA^r(\bk)$. Up to conjugacy, the above morphism only depends on the image $[\lambda_1:\cdots:\lambda_r]$ in $\mP^{r-1}=\mP^{r-1}(\bk)$, so we will usually reserve the notation $\ulambda$ for this element in $\mP^{r-1}$. The following is \cite[Lemma~3.2.2]{CF2}.

\begin{lemma}\label{Lem:Moore}
    The following conditions are equivalent on 
    $\ulambda=[\lambda_1:\lambda_2:\cdots:\lambda_r]$ in $\mP^{r-1}(\bk)$:
    \begin{enumerate}
        \item The homomorphism $\iota_{\ulambda}$ is injective.
        \item The elements $\lambda_i\in\bk$ are linearly independent over $\mF_p$.
        \item The determinant of the Moore matrix 
        $$ \mathbb{M}(\ulambda):=\begin{pmatrix}
\lambda_1 & \lambda_2 & \lambda_3 & \cdots & \lambda_r  \\
\lambda_1^p & \lambda_2^p & \dots \\
\dots \\
\lambda^{p^{r-1}}_1 & \dots&&&\lambda_r^{p^{r-1}}
\end{pmatrix}$$
is non-zero.
    \end{enumerate}
    We denote the set of $\ulambda$ satisfying these equivalent properties by $\mP_f^{r-1}\subset\mP^{r-1}$.
\end{lemma}
 Following \cite[\S 4.2]{CF2}, for $\ulambda\in\mP_f^{r-1}$, we can define $\blambda\in\mP^{r-1}$, uniquely determined by
    $$\sum_{i=1}^r\lambda_i^{p^l}\blambda_i=0,\qquad\mbox{for all}\quad 0\le l\le r-2.$$

\begin{lemma}\label{lem:bijP}
    The assignment $\ulambda\mapsto\blambda$ gives a bijection $\mP_f^{r-1}\to\mP_{f}^{r-1}$.
\end{lemma}
\begin{proof}
For $r=1$ the assertion is immediate, so assume $r\ge2$. First we need to establish that $\blambda\in\mP_f^{r-1}$. If we choose representatives in $\mA^r$, then $\blambda$ spans the orthogonal complement of $\ulambda, \ulambda^p,\ldots, \ulambda^{p^{r-2}}$, where the powers are taken component-wise. If $\blambda\not\in\mP_f^{r-1}$ then there are $a_i\in\mF_p$, not all zero, with $\sum_i a_i\blambda_i=0$, which thus in turn means that $(a_1,\ldots, a_r)\in\mA^r$ is in the linear span of $\ulambda, \ulambda^p,\ldots, \ulambda^{p^{r-2}}$. In other words, there are $c_i\in\bk$, not all zero, such that
\[\sum_{j=0}^{r-2}c_j\ulambda^{p^j}\;\in\;\mA^{r}(\mF_p).\]
By taking the $p$-th power, this means that
\[\sum_{j=0}^{r-2}c_j^p\ulambda^{p^{j+1}}\;=\;\sum_{j=0}^{r-2}c_j\ulambda^{p^j},\]
and by linear independence of the $\ulambda^{p^j}$ for $0\le j\le r-1$ it follows quickly that this forces $c_j=0$ for all $j$, a contradiction.

    The defining relations for $\blambda$ can be rewritten as $\sum_{i=1}^r\lambda_i^{p^{r-2}}\blambda_i^{p^j}=0$ for $0\le j\le r-2$. In other words, applying the procedure $\ulambda\mapsto \blambda$ to $\blambda$ produces $\ulambda^{p^{r-2}}$. Hence $\ulambda\mapsto \blambda$ has a two-sided inverse.
\end{proof}

\subsubsection{Subgroups of symmetric groups}

We let $E_r<S_{p^r}$ denote the copy of $C_p^r$ embedded as in \cite[\S 5.1.5]{CF1}. It is uniquely determined up to conjugacy due to the following well-known lemma.
\begin{lemma}\label{lem:ESn}
    If $n>1$ and $S_n$ has a transitive elementary abelian $p$-subgroup $E$ then $n=p^r$ for some $r\in\mZ_{>0}$ and $E$ is conjugate to $E_r$.
\end{lemma}
\begin{proof}
    It is clear that any transitive abelian subgroup of $S_n$ must act regularly on $\{1,\ldots, n\}$. In particular $E=C_p^r$ with $n=p^r$. 

    So now consider two transitive elementary abelian $p$-subgroups $E,F$ of $S_n$ and an isomorphism $\varphi:E\to F$. Define $\sigma\in S_n$ via
    \[\sigma(e(1))=\varphi(e)(1),\]
    for all $e\in E$. That $\sigma$ is a permutation follows from regularity of the actions. Now we can observe that $F=\sigma E\sigma^{-1}$.
\end{proof}

\subsubsection{Support theory}
We refer to \cite{BensonBook} for the basic notions of cohomological support theory and the rank variety for elementary abelian $p$-groups.

Set $X_i=g_i-1\;\in\; \bk C_p^r.$
 For $M\in \Rep C_p^r$, we define its support $\supp(M)\subset\mP^{r-1}$ as the projective variety that has as cone the rank variety $V^{\mathrm{rk}}(M)\subset\mA^r$. Concretely, the point $[\mu_1:\mu_2:\cdots:\mu_{r}]\in \mP^{r-1}$ is in $\supp(M)$ if and only if $M$ is not projective as a module over the algebra $\bk[t]/t^p$, where $t$ acts on $M$ as
 $$\sum_{i=1}^r\mu_i X_i\;\in\; \bk C_p^r.$$
Recall that
\begin{equation}\label{sup:tensor}
\supp(M\otimes N)\;=\;\supp(M)\cap \supp(N),
\end{equation}
and that
$M\in\Rep C_p^r$ is projective if and only if $\supp(M)=\emptyset$.

\begin{lemma}\label{lem:support}
    If $M\in \Rep C_p^r$ is not projective, but is projective over every proper subgroup of $C_p^r$, then $\supp(M)\cap\mP^{r-1}_f\not=\varnothing$.
\end{lemma}
\begin{proof}
Subgroups $C_{p}^{r-1}\cong G<C_p^r$ are in natural bijection with codimension one $\mF_p$-subspaces of $\mF_p^r$. So for every such $G$ we have a corresponding $a^G\in \mP^{r-1}(\mF_p)$ describing the hyperplane. In particular, each such subgroup $G$ determines an $\mF_p$-rational hyperplane $P_G\subset\mP^{r-1}$ comprising all $\ulambda$ with $\sum_i a_i^G\lambda_i=0$. Moreover, by \cite[Lemma~1.9.8]{BensonBook} we have
\[\supp(M|_G)\;=\;\supp(M)\cap P_G.\]

Since the complement of the unions of the $P_G$ is precisely $\mP_f^{r-1}$, the condition that $M$ be projective over all proper subgroups implies $\supp(M)\subset \mP_f^{r-1}$ (see \cite[Lemma~1.9.5]{BensonBook}), while the condition that $M$ not be projective shows $\supp(M)\not=\varnothing$, concluding the proof.
\end{proof}

\subsection{Tensor categories}

\subsubsection{Definitions} \label{sec:def-tc}

We define \emph{tensor categories} to be essentially small finite-length abelian symmetric rigid monoidal $\kk$-linear categories with one-dimensional endomorphism space of the tensor unit. In other words, they are symmetric tensor categories in the sense of \cite{EGNO}. \emph{Tensor functors} are exact linear symmetric monoidal, and automatically faithful. A tensor category $\cC$ is said to be \emph{of moderate growth} if for any $X\in\cC$ there is $K\in\mathbb{R}$ such that $\len(X^{\o n})\le K^n$, where $\len$ denotes the length of an object in an abelian category. A tensor category is \emph{finite} if it has enough projectives and only finitely many simple objects up to isomorphism.

\begin{remark}\label{rem:twistlin}
    We will often encounter exact symmetric monoidal functors $F:\cC\to\cD$ between tensor categories that are not $\bk$-linear, but `twisted' $\bk$-linear by a field automorphism. In practice this automorphism will be a Frobenius homomorphism $\lambda\mapsto \lambda^q$ for a power $q$ of $p$, meaning $F(\lambda f)=\lambda^q F(f)$ for a morphism $f$ in $\cC$ and $\lambda\in\bk$. Since $\bk$ is perfect, we can twist the $\bk$-linear structure on the target, to obtain a new tensor category (with same underlying abelian monoidal category) for which $F$ is now an honest tensor functor. We will therefore largely ignore this subtlety and freely treat these functors as tensor functors.
\end{remark}

In a tensor category with finitely many simple objects up to isomorphism, for any object $X$, the exact functor $X\o -$ corresponds to multiplication with an integer matrix in the Grothendieck group. By the Frobenius--Perron theorem, the spectral radius of this matrix is an eigenvalue. The \emph{Frobenius-Perron dimension} $\FPdim(X)\in\mathbb R$ of $X$ is defined as this maximal eigenvalue. $\FPdim$ is additive under direct sums, multiplicative under tensor products, and preserved by tensor functors, see~\cite{EGNO}. As $\len(X^{\o n})\le \FPdim(X)^n$, tensor categories with finitely many isomorphism classes of simples, and in particular finite tensor categories, are automatically of moderate growth.

Since we are working over an algebraically closed field, for any two tensor categories $\cC,\cD$ in our sense, there is a \emph{Deligne tensor product} $\cC\boxtimes\cD$ (see \cite[Section~3]{CF1} and references therein for more information). It is a tensor category containing the naive tensor product category $\cC\o_\kk\cD$, whose objects are pairs of objects from $\cC$ and $\cD$, and whose hom-spaces are tensor products of hom-spaces in $\cC$ and $\cD$, respectively, over $\kk$. In fact, any object in $\cC\boxtimes\cD$ is a quotient of an object in $\cC\o_\kk\cD$, and the simple objects of $\cC\boxtimes\cD$ are exactly the objects of the form $X\boxtimes Y=(X,Y)$, for simple objects $X\in\cC$, $Y\in\cD$. If $\cC$ is a finite tensor category, then we can realise $\cC\boxtimes\cD$ as the category of $\bk$-linear functors $\cP_\cC^\op\to\cD$, where $\cP_\cC\subset\cC$ is the full subcategory of projective objects. In this realisation, an object of the form $X\boxtimes Y$ corresponds to the functor $\Hom(-,X)\o Y$.


\subsubsection{Higher Verlinde categories}

We recall the construction of the higher Verlinde categories $\Ver_{p^r}$ for $r\ge1$. Let $\Tilt\SL_2$ be the category of tilting modules of $\SL_2$ over $\kk$. Its indecomposable objects $T_i$ can be characterised by their highest weights $i\ge0$. Then $\Ver_{p^r}$ is the monoidal abelian envelope of the quotient category $\Tilt\SL_2/I_r$, where $I_r$ is the tensor ideal of morphisms that factor through a direct sum of tilting modules $T_i$ with $i\ge p^r-1$. This ideal $I_r$ is also the tensor ideal generated by the Steinberg representation $T_{p^r-1}$. As a special case, when $r=1$, $\Tilt\SL_2/I_1$ is already semisimple, and hence its own abelian envelope, with simples $L_1,\dots,L_{p-1}$, where $L_i$ is the image of $T_{i-1}$, following the notation from \cite{Os}, rather than \cite{BEO}.

\subsubsection{\texorpdfstring{$\bV$}{V}-functors and \texorpdfstring{$\bO$}{O}-functors}\label{sec:VOfun}

We will freely use that for a tensor category $\cC$ and $X,Y\in\cC$, the braiding produces algebra morphisms $\bk S_n\to\End(X^{\otimes n})$, so that $\Hom(Y,X^{\otimes n})$ becomes a $S_n$-representation via composition.

We recall some definitions from \cite{CF1}. For a permutation group $G$, a \emph{$\bV$-functor} is a linear symmetric monoidal functor of the form
$$
\Theta:\Rep G\to\cV ,
$$
for some tensor category $\cV$, that annihilates objects of the form $\Ind_H^G\unit$, for all non-transitive subgroups $H<G$.

Now assume that the permutation group $G$ is defined as a subgroup of $S_{p^r}$, for some $r\ge1$. To any $\bV$-functor $\Theta$ as above, and any tensor category $\cC$, we can associate the additive symmetric monoidal functor 
$$
\Phi_{\cC}\: \cC
\xrightarrow{X\mapsto X^{\o p^r}}
    \cC\boxtimes\Rep S_{p^r}
\xrightarrow{\cC\boxtimes\Res} 
    \cC\boxtimes\Rep G
\xrightarrow{\cC\boxtimes\Theta}
    \cC\boxtimes\cV .
$$
Here, we view the $p^r$-th tensor power as an object in $\cC$ with an action of $S_{p^r}$ by permuting tensor factors. 
It was shown in \cite[Section~5]{CF1} that such functors intertwine tensor functors in a certain way, as a generalisation of this property for Ostrik's Frobenius functor in \cite{Os, EOf, tann}. The term $\bO$-functor was coined in \cite{CF1} for such functors. Such functors need not be exact, but when they are, they are tensor functors up to the comments in Remark~\ref{rem:twistlin}.

Assume that $\cC$ is a finite tensor category with subcategory $\cP_\cC$ of projective objects. Then using the description from \Cref{sec:def-tc}  of the Deligne tensor product $\cC\boxtimes\cV$ as a functor category $\cP_{\cC}^{\op}\to \cV$, we can think of the above $\bO$-functors alternatively as functors 
\begin{equation}\label{eq:reinterpretPhi}
\Phi_{\cC}:\cC\to \cC\boxtimes \cV, \quad
X\mapsto \Theta\circ\Hom(-,X^{\o p^r}).
\end{equation}
Here $\Hom(-,X^{\o p^r})$ is viewed as a functor $\cP_\cC^{\op}\to\Rep G$, using the permutation action of $G<S_{p^r}$ on tensor factors. This more explicit form unfortunately obscures the monoidal structure.

As a special case, take $G=C_p<S_p$, and take $\Theta\:\Rep G\to\Ver_p$, the semisimplification functor. Then the above construction recovers Ostrik's Frobenius functor
\begin{equation}\label{eq:classicFr}\Fr:\cC\to \cC\boxtimes\Ver_p,\quad X\mapsto \Fr(X)=\bigoplus_{i=1}^{p-1}F_i(X)\boxtimes L_i,\end{equation}
where $F_i(X)=\frac{\ker(d)\cap\im(d^{i-1})}{\ker(d)\cap\im(d^i)}$, for $d$ the action of a generator of $C_p$ minus the identity on $X^{\o p}$. Using the same formula, we also set $F_p(X)=\im(d^{p-1})$.

Given an $\bO$-functor $\Phi$, we say that a tensor category $\cC$ is $\Phi$-exact when $\Phi_{\cC}$ is exact, and thus a tensor functor. In particular $\Fr$-exactness is the same as `Frobenius exactness'.


\section{The refined Steinberg splitting conjecture} \label{sec:St}

In this section, we prove \Cref{eq:wet-spaghetti} and \Cref{conj:CF}. The proof strategy is as follows: First, it will turn out to be useful to consider all (not just $S=St_{r-1}$ in the statement) representations of elementary abelian $p$-groups as representations of the additive group. Specifically, by this change of perspective, we will be able to use certain operators (called $(c_i)_i$ in \Cref{sec:St-Cpr}) that are very natural for representations of the additive group. These operators define a generalised cohomology on representations, along which tensor products with (lifts of) Steinberg representations are shown to split in \Cref{sec:St-poly}. This implies \Cref{conj:CF}.

In more detail, the statement of \Cref{conj:CF} requires (up to conjugacy) an embedding $C_p^r<(\bk,+)<\SL_2(\bk)$, see \ref{sec:inSL2}. We thus have resulting monoidal functors
\[\Tilt \SL_2\to\Rep\mG_a\to\Rep C_{p}^r.\]
Denote by $\mG_{a,j}$ the $j$-th Frobenius kernel of $\mG_a$. The group algebra is given by \begin{equation}\label{eq:kGar}\bk \mG_{a,j}\cong\bk[z_0,z_1,\ldots, z_{j-1}]/(z_0^p,\ldots,z_{j-1}^p)\cong \bk\mG_{a,j-1}\otimes\bk[z_{j-1}]/z_{j-1}^p.\end{equation}
We can filter $\Rep\mG_a$ by full (non-monoidal) subcategories 
\[\Vect=(\Rep\mG_a)_0\subset (\Rep\mG_a)_1\subset (\Rep\mG_a)_2\subset\cdots,\]
where $(\Rep\mG_a)_j$ contains the representations that have polynomial degree less than $p^j$. The monoidal restriction functors give rise to $\bk$-linear equivalences
\[\Rep \mG_{a,r}\xleftarrow{\sim}(\Rep\mG_a)_r\xrightarrow{\sim}\Rep C_p^r.\]
Moreover, the restriction of $St_{r-1}$ to $\mG_a$ corresponds to the projective cover of $\unit$ in $(\Rep\mG_a)_{r-1}$. 

To understand tensor products with $St_{r-1}$ in $\Rep C_p^r$ it thus suffices to understand tensor products of objects in $(\Rep\mG_a)_r$ with $St_{r-1}$, and these are objects in $(\Rep\mG_a)_{r+1}$. Subsequently, it is sufficient to understand tensor products in $\Rep \mG_{a,r+1}$ of objects in $\Rep \mG_{a,r}$, where we have a non-monoidal inclusion $\Rep \mG_{a,r}\subset\Rep \mG_{a,r+1}$ by letting $z_r$ act as zero, see \eqref{eq:kGar}, with $\bk\mG_{a,r-1}$. 

Using \eqref{eq:kGar}, an obvious class of $\mG_{a,r}$-representations are given by modules of the form
\[M=\bigoplus_{i=1}^p M_i\boxtimes J_i\]
where $M_i\in\Rep\mG_{a,r-1}$ and $J_i$ is the indecomposable $\bk[z_{r-1}]/z_{r-1}^p$-module of dimension $i$. In general, for $V\in\Rep\mG_{a,r}$, we can set 
$$
W_j(V) := \ker (z_{r-1}|_V) \cap \im ((z_{r-1}|_V)^{j-1}) ,\quad
B_j(V) := W_j(V)/ W_{j+1}(V) 
\quad \in \Rep\mG_{a,r-1}.
$$
Clearly $B_j(M)\cong M_j$, for $M$ as above. The main technical result (\Cref{prop:splitting}) is now that 
\[\bk \mG_{a,r-1}\otimes V\;\cong\; \bigoplus_{j=1}^p\bk \mG_{a,r-1} \otimes (B_j(V)\boxtimes J_j)\]
 inside $\Rep\mG_{a,r+1}$. Using Nakayama's lemma it is actually sufficient to establish a morphism that is an isomorphism modulo the Jacobson radical. Constructing this morphism is the difficult part, and relies on a brute force construction.

\subsection{Normal form of Steinberg tensor products} \label{sec:normal-tp} Our first goal is to derive, in \Cref{cor:T}, a description of a functor given by tensoring with representations that, under a suitable identification, will correspond to Steinberg representations.

\begin{definition} Let $\Pol:=\Rep\mG_a$ be the tensor category of finite-dimensional rational (or `polynomial') representations over $\kk$ of the additive group $\mG_a=(\bk,+)$. We think of its objects\footnote{Here, a pair $(V,g)$ corresponds to the coaction of $\kk[\mG_a]=\kk[x]$ given by $v\mapsto \sum_{i\ge0}g_i(v)\o x^i$. On $K$-points, where $K$ is any (possibly infinite-dimensional) commutative $\kk$-algebra, it corresponds to the action of $(K,+)$ given by $\alpha\mapsto g(\alpha)$.} as pairs $(V,g)$, where $V$ is a vector space over $\kk$ and $g=\sum_{i\ge0} g_i T^i$ is an element in $\End_\kk(V)[T]$, i.e., a matrix polynomial in one indeterminate $T$, satisfying
$$
g(0)=1=\id_V ,\quad
g(T+S)=g(T) g(S) ,
$$
as polynomials, where $S$ is another indeterminate.
The \emph{degree} of a polynomial representation $(V,g)$ is the degree $\deg(g)$ of $g$ in $T$. For any $r\ge0$, let $\Pol_r=(\Rep\mG_a)_r$ be the full subcategory of $\Pol$ on objects $(V,g)$ of degree less than $p^r$.
\end{definition}

In the above description, the tensor product of $(V,g)$ and $(W,h)$ is given by $(V\otimes W,g\otimes h)$, where $g\otimes h\in \End(V\otimes W)[T]$ is based on the Kronecker product of matrices. 

At the same time, polynomial representations as in the definition are nothing but tuples of commuting $p$-nilpotent operators\footnote{... which essentially means thinking of them as $\operatorname{Dist}\mG_a$-modules.}: For any vector space $V$ and $x=(x_0,\dots,x_{r-1})\in\End_\kk(V)^r$ a tuple of pairwise commuting $p$-nilpotent operators, define
$$
g^{\{x\}} := \prod_{0\le i\le r-1} \exp(x_i T^{p^i})\quad
\in\End_\kk(V)[T] ,     
\quad
\text{with }
\exp(x) := \sum_{0\le k\le p-1} \frac{x^k}{k!} .
$$
We will also use the notation
$$
x^{[a]} := \prod_{0\le i\le r-1} \frac{x_i^{a_i}}{a_i!} \quad
\tforall 0\le a< p^r ,
$$
using the $p$-adic expansion $a=a_0+pa_1+\dots+ p^{r-1} a_{r-1}$. We set $x^{[a]}=0$ for $a\ge p^r$. In particular,
$$
g^{\{x\}} = \sum_{0\le a<p^r} x^{[a]} T^a
$$
Set 
$$
A_r := \kk[z_0,\dots,z_{r-1}]/(z_0^p,\dots,z_{r-1}^p) ,
$$
let $A_r\lmod$ denote its category of finite-dimensional modules. Its objects can be thought of as pairs $(V,x)$, where $V$ is a vector space and $x$ is an $r$-tuple of pairwise commuting $p$-nilpotent operators on $V$ as before. 

In the following, $\Psi_r$ is the restriction $\Rep\mG_a\to\Rep\mG_{a,r}$ to the Frobenius kernel, expressed in terms of the algebra $A_r$ (which can be identified with the group algebra of $\mG_{a,r}$).

\begin{lemma} \label{lem:equivalences} 
For any $r\ge0$, there are exact functors 
$$
\Psi_r\:\Pol\to A_r\lmod, \quad
(V,g) \mapsto (V,(g_{p^0},\dots,g_{p^{r-1}})) ,
$$
$$
\text{and}\quad
\Psi^-_r\: A_r\lmod\to\Pol, \quad (V,x)\mapsto (V,g^{\{x\}}), 
$$
that both commute with the faithful linear forgetful functors to vector spaces and 
that induce an equivalence of linear categories $\Pol_r \xrightarrow\simeq A_r\lmod$.
\end{lemma}

\begin{proof} Starting with a polynomial representation $(V,g)$, additivity of $g$ means that for all $i,j\ge0$,
$$
g_i g_j
= [g(T)g(S)]_{T^i S^j}
= [g(T+S)]_{T^i S^j} 
= g_{i+j} \binom{i+j}{i} .
$$
Since the right-hand side is symmetric with respect to swapping $i\leftrightarrow j$, this shows that all $g_i$ commute pairwise. Also, for any $m\ge0$,
$$
g_{p^m}^p = g_{p^{m+1}} (\dots) \binom{p^{m+1}}{p^m} = 0 ,
$$
omitting more binomial coefficients, and using Lucas' theorem for the last one. Hence, $g_{p^m}$ is $p$-nilpotent, implying that $\Psi_r$ is well-defined. The first displayed formula also implies that for any $m\ge0$ and $i=\sum_{k\ge0} i_k p^k$ with $i_m>0$,
$$
g_{p^m} g_{i-p^m} =g_i \binom{i}{p^m} = g_i i_m ,
$$
again using Lucas' theorem. This implies that, if $\deg(g)<p^s$, then $g = g^{\{x\}}$ for $x=(g_{p^0},\dots,g_{p^{s-1}})$. In particular, for $(V,g)\in\Pol_r$ we may take $s=r$.

For any tuple $x$ of commuting $p$-nilpotent operators, $g^{\{x\}}(0)=1$ and additivity of $g^{\{x\}}$ in the indeterminate follows from that of the polynomials $\exp(x_i T^{p^i})$, implying $\Psi^-_r$ is well-defined. If $g=g^{\{x\}}$, then indeed $g_{p^i}=x_i$, completing the proof. 
\end{proof}

For any $r\ge1$, note that the subcategories $\Pol_r\subset\Pol$ fail to be closed under tensor products. We want to control this failure at least for tensor products involving the regular representation in $A_{r-1}\lmod$, i.e., under the above equivalence, the polynomial representation
$$
(A_{r-1}, g^{\{z\}})\in \Pol_{r-1}\subset\Pol_r ,\quad
\text{with } z:=(z_0,\dots,z_{r-2}) ,
$$
viewed as a tuple of commuting $p$-nilpotent endomorphisms of $A_{r-1}$. The offset in the indices is intended; working with $A_{r-1}$ in $\Pol_r$ will lead us to the phenomenon relevant for our purposes. $A_{r-1}$ will correspond to a (higher) Steinberg representation under a suitable identification. 

We start with some auxiliary notions.

\begin{definition} For any vector space $V$, $r\ge1$, and $(r-1)$-tuples $x,y$ of commuting $p$-nilpotent operators on $V$ such that operators from $x$ and from $y$ always commute pairwise, let $x*y$ and $C(x,y)$ be the first $(r-1)$ entries and the last entry, respectively, of the $r$-tuple corresponding by \Cref{lem:equivalences} to the polynomial representation $(V,g^{\{x\}} g^{\{y\}})$, i.e.~such that
\begin{equation}
g^{\{x\}} g^{\{y\}} = g^{\{x*y\}} \exp(T^{p^{r-1}} C(x,y)) . 
\label{eq:star}
\end{equation}
\end{definition}

The above symbolically means
$$
\Psi_r((V, g^{\{x\}}g^{\{y\}})) = (V, (x\ast y, C(x,y))) ,
$$
using the concatenation of tuples on the right-hand side.
Note that the product $g^{\{x\}} g^{\{y\}}$ has degree less than $p^r$, as required.

\begin{lemma} \label{lem:involution} For all $x,y$ as above, $C(x,y)^2=0$, and we have 
$$
(x*y)*(-y) = x ,\quad
C(x*y,-y) = -C(x,y) .
$$    
\end{lemma}

\begin{proof} If $p=2$, then $C(x,y)^2=0$ immediately. Otherwise, $C(x,y)$ has to square to zero as $g^{\{x\}} g^{\{y\}}$ is of degree less than $2p^{r-1}$.
Multiplying \Cref{eq:star} with $g^{\{-y\}} \exp(-T^{p^{r-1}} C(x,y))$ and using 
$$
g^{\{y\}} g^{\{-y\}} = 1 = \exp(T^{p^{r-1}}C(x,y)) \exp(-T^{p^{r-1}}C(x,y))
$$
implies the remaining assertions.
\end{proof}

\begin{proposition} \label{lem:tensor-product} For any $r\ge1$ and $(V,g)\in\Pol_r$ the tensor product representation on $A_{r-1}\o V$ is isomorphic as a polynomial representation to the one given by
$$
g^{\{x\}} \exp(T^{p^{r-1}} D_1) (1+T^{p^r} D_2) ,
$$    
using the definitions 
$$
x:=(z_i\o\id)_{0\le i\le r-2} ,\quad
y:=(\id\o g_{p^i})_{0\le i\le r-2}, \quad
R:=-C(x, -y) ,
$$
$$
D_1:= (\id\o g_{p^{r-1}})+R
,\quad
D_2:= -(\id\o g_{p^{r-1}})^{p-1} R
\quad\in\End_\kk(A_{r-1}\o V) .
$$
\end{proposition}

\begin{proof} By definition, and using \Cref{lem:equivalences} and \Cref{lem:involution}, the tensor product polynomial representation on $A_{r-1}\o V$ is given by
$$
G(T) := g^{\{x\}} g^{\{y\}} \exp(T^{p^{r-1}} (\id\o g_{p^{r-1}}))
= g^{\{x*y\}} (1+T^{p^{r-1}} C(x,y)) \exp(T^{p^{r-1}} (\id\o g_{p^{r-1}}))
.    
$$

Consider the linear endomorphisms $\phi_\pm$ of $A_{r-1}\o V$ given by
$$
\phi_\pm( z^{[a]}\o v) = (x*(\pm y))^{[a]} (1\o v) \quad
\tforall  0\le a< p^{r-1}.
$$
A computation using \Cref{lem:involution} shows that $\phi_-=\phi_+^{-1}$. Conjugation with $\phi_-$ turns $G(T)$ into
$$
g^{\{(x*(-y))*y\}} (1+T^{p^{r-1}} C(x*(-y),y)) \exp(T^{p^{r-1}} (\id\o g_{p^{r-1}}))
= g^{\{x\}} (1+T^{p^{r-1}} R) \exp(T^{p^{r-1}} (\id\o g_{p^{r-1}}))
,
$$
again using \Cref{lem:involution}.
For any two commuting operators $P_1,P_2$ with $P_1^2=0=P_2^p$ and an indeterminate $S$, we compute
$$
(1+S P_1)\exp(S P_2) = \exp(S(P_1+P_2)) (1-S^p P_1 P_2^{p-1}) ,
$$
using $(p-1)!=-1$. This shows that the conjugate of $G(T)$ above is the asserted polynomial.
\end{proof}

We can reformulate \Cref{lem:tensor-product} as follows:

\begin{corollary} \label{cor:T} Set
$
\ol A_{r+1} := A_{r+1}/(z_{r-1} z_r,z_r^2)
$.
There is a diagram of exact functors, commutative up to natural isomorphism,
$$
\begin{tikzcd}
A_r\lmod 
\ar[r, "\simeq"] 
\ar[d, "\cT", swap] 
& \Pol_r 
\ar[r, "A_{r-1}\o-"] 
& \Pol_{r+1} 
\ar[d, "\simeq"]
\\
 \ol A_{r+1}\lmod 
\ar[rr]
&& A_{r+1}\lmod
\end{tikzcd} ,
$$
with the equivalences from \Cref{lem:equivalences}, the natural inclusion functor between the module categories of $\ol A_{r+1}$ and $A_{r+1}$, and a functor $\cT$ given by
$$
\cT(X)= A_{r-1}\o X,
$$
where $z_0,\dots,z_{r-2}$ act only on the first tensor factor and the action of $z_{r-1}, z_r$ is given by $D_1,D_2$ as in \Cref{lem:tensor-product}.
\end{corollary}

\subsection{Steinberg splitting for polynomial representations} \label{sec:St-poly} Our goal is to derive an even more concrete description of the functor $A_{r-1}\o-$ on $\Pol_r\simeq A_r\lmod$, using the information from \Cref{cor:T} as a starting point. More precisely, we will see that objects of the form $\cT(X)$ as above can be decomposed as direct sums according to a suitable generalised cohomology on the objects $X$. This direct sum decomposition will be in the full subcategory on objects which are images under $\cT$. Hence, \Cref{cor:T} tells us that the decomposition can be detected and constructed in $\ol A_{r+1}\lmod$.

We fix $r\ge1$ and set $A:=A_{r-1}$. To define the modules in the sought after decomposition and describe morphisms to them, we need some notation. Let 
$$
A\:\Vect\to A\lmod
$$
be the free $A$-module functor, i.e., $A(V)=A\o V$ with $A$ acting on the first factor. We abuse notation by using the same symbol for a $\kk$-algebra and a functor.
For any $A$-module $X$ (in particular, for the restriction of an $A_r$-module), set $x:=(z_i|_X)_{0\le i\le r-2}$ and
$$
R_X := -C((z_i\o\id)_i, (-\id\o x_i)_i) = [-g^{\{z\}}(T)\o g^{\{x\}}(-T)]_{T^{p^{r-1}}} \quad\in \End_A(A(X)) ,
$$
which is the operator $R$ from \Cref{lem:tensor-product}.
For any $N\in A\lmod$ and any $1\le j\le p$, set $J_j:=\kk v_1\oplus\dots\oplus\kk v_j$ and let 
$$
N \boxtimes J_j \quad\in A_r\lmod
$$
be the vector space $N\o J_j$ with  $z_0,\dots,z_{r-2}$ acting on the first factor and $z_{r-1}$ acting on the second factor $J_j$ as a nilpotent Jordan block, by sending $v_j\mapsto 0$ and $v_k\mapsto v_{k+1}$ for $1\le k\le j-1$.

Our goal will be to show that objects of the form $\cT(Y)$ for $Y\in A_r\lmod$ are direct sums of objects of the form $\cT(N\boxtimes J_j)$, for suitable $N\in A\lmod$. Our first step is hence to parametrize morphisms between these kinds of objects. The idea is that such morphisms can be written with coefficient maps for $v_1,\dots,v_j$. Intertwining the action of $z_0,\dots,z_r$ is equivalent to a system of concrete conditions on these coefficient maps that can be solved explicitly. 

\begin{lemma} \label{lem:jordan}
For $Y\in A_r\lmod$, $B\in A\lmod$, and $1\le j\le p$, we define
$$
H(Y,B,j) := \{f: Q^j (f)=0 = f R_Y A((z_{r-1}|_Y)^{p-1}) - R_B Q^{p-1} (f)\} \quad
\subset \Hom_A(A(Y), A(B)) ,
$$
with $Q(f):=f A(z_{r-1}|_Y) + f R_Y - R_B f$. Then we have an isomorphism
$$ \Hom_{\ol A_{r+1}}(\cT(Y), \cT( B\boxtimes J_j) ) \cong H(Y,B,j) , \quad
h\mapsto \pi h ,
$$
where $\pi\:A\o B\o J_j\to A\o B$ projects to the coefficient of $v_j$.
\end{lemma}

\begin{proof} Set $M:=\cT(Y)$ and $N:=\cT(B\boxtimes J_j)$. By \Cref{lem:tensor-product}, $M$ is an $\ol A_{r+1}$-module with underlying vector space $A\o Y$, with $z_0,\dots,z_{r-2}$ acting on the first factor $A$, and with
$$
z_{r-1}|_M = R_Y + A(z_{r-1}|_Y) ,\quad
z_r|_M = -R_Y A(z_{r-1}|_Y)^{p-1} .
$$
Similarly, $N$ is an $\ol A_{r+1}$-module with underlying vector space $A\o B\o J_j$, with $z_0,\dots,z_{r-2}$ acting on the first factor $A$, and with
$$
z_{r-1}|_N = R_B\o\id + \id\o\id\o d ,\quad
z_r|_N = -R_B\o d^{p-1} ,
$$
where $d$ is the action of the nilpotent Jordan block on $J_j$ in its chosen basis.
For any $f\in\Hom_\kk(M,N)$, we can write
$$
f = \sum_{1\le a\le j} f_a \o v_a \quad
\text{with } f_a \in \Hom_\kk(M,A\o B) .
$$
$f$ intertwining the action of $z_0,\dots,z_{r-2}$ is equivalent to $f_a\in\Hom_A(M,A(B))$ for all $a$. $f$ intertwining the action of $z_{r-1}$ is equivalent to
$$
f_1 z_{r-1}|_M = R_B f_1 ,\quad
f_a z_{r-1}|_M = R_B f_a + f_{a-1} \quad
\tforall 2\le a\le j ,
$$
which is equivalent to $f_a=Q^{j-a}(f_j)$, for all $a$, and $Q^j(f_j)=0$, which is one of the conditions in the definition of $H(Y,B,j)$. The maps $f$ described by such coefficient maps $(f_a)_a$ are $A$-linear if and only if $f_j$ is. For them, intertwining the $z_r$-action is equivalent to
$$
f_a z_r|_M = \begin{cases} 0 & a<p \\ -R_B f_1 = -R_B Q^{p-1} (f_j) & a=p \end{cases} \quad
\tforall 1\le a\le j .
$$
As $Q(\cdot)$ commutes with $(\cdot z_r|_M)$ and with left multiplication by $R_B$, and $Q^j(f_j)=0$, this is equivalent to
$$
f_j R_Y A(z_{r-1}|_Y)^{p-1} = \begin{cases} 0 = R_B Q^{p-1} (f_j) & j<p \\ R_B Q^{p-1} (f_j) & j=p \end{cases} .
$$
\end{proof}

Recall that for any $Y\in A_r\lmod$, we want to show that $\cT(Y)$ is direct sum of objects of the form $\cT(B\boxtimes J_j)$. More precisely, the object $B$ corresponding to an index $j$ should be given by a certain generalized cohomology $B_j(Y)$, as follows. The generalised cohomologies $B_j$ were studied in \cite[Section~3]{EOf}.

\begin{definition} \label{def:B} For $Y\in A_r\lmod$, we define for all $j\ge1$: 
$$
W_j(Y) := \ker (z_{r-1}|_Y) \cap \im ((z_{r-1}|_Y)^{j-1}) ,\quad
B_j(Y) := W_j(Y)/ W_{j+1}(Y) 
\quad \in A_{r-1}\lmod=A\lmod.
$$
\end{definition}

To obtain the projections $\cT(Y)\tto \cT(B_j(Y)\boxtimes J_j)$, we want to use \Cref{lem:jordan}. Hence, we need an element in the parametrising space $H(Y,B_j(Y),j)$. We can build such an element by starting with a morphism of the form $A(\beta)$, where $\beta\: Y\to B_j(Y)$ is a choice of a linear morphism extending the defining quotient map $W_j(Y)\tto B_j(Y)$. $A(\beta)$ itself might not lie in $H(Y,B_j(Y),j)$, but we will find a correction morphism $\ell$ that achieves this, in the following sense. 

\begin{lemma} \label{lem:correction} Fix $1\le j\le p$. For $Y\in A_r\lmod$, set $d:=z_{r-1}|_Y$. Assume that there is a linear map $\beta\:Y\to B:=B_j(Y)$ such that
\begin{equation}
\beta|_{W_j(Y)} = (W_j(Y)\tto B) ,\quad
\beta d^j = 0 ,\quad
R_B A(\beta) R_Y = 0 .
\label{eq:correction-beta}    
\end{equation}
Assume also that there is an $A$-linear map $\ell\:A(Y)\to A(B)$ such that
\begin{equation}
\ell R_Y = R_B \ell ,\quad
\ell A(d^j) = A(\beta j d^{j-1}) R_Y - R_B A(\beta j d^{j-1}) .
\label{eq:correction-ell}    
\end{equation}
Then $\eta:=A(\beta)-\ell$ lies in $H(Y,B,j)$.
\end{lemma}

\begin{proof}
Set
$$
\delta(h) := h R_Y - R_B h ,\quad
Q(h) := h A(d) + \delta(h), 
\quad \tforall h\in\Hom_A(A(Y), A(B)) .
$$
Note that $Q$ as defined here agrees with $Q$ as in \Cref{lem:jordan}.

As $R_Y^2$ and $R_B^2$ are zero by \Cref{lem:involution}, using the assumptions we get
$$
\delta^2(\eta) = \delta^2(A(\beta)) = 0 .
$$
As $\delta(\cdot)$ and $(\cdot A(d))$ commute, this implies
$$
Q^j(\eta) = (\delta(\cdot) + (\cdot A(d)))^j (\eta)
= \eta A(d^j) + \delta(A(\beta)) A(j d^{j-1})
= A(\beta d^j) = 0 ,
$$
using the assumptions in the last two steps. Similarly, we compute
$$
R_B Q^{p-1}(\eta) = R_B \eta A(d^{p-1}) - R_B \delta(A(\beta)) A(d^{p-2}) 
 = \eta R_Y A(d^{p-1})
$$
using $\delta(\eta) A(d^{p-1})=0$ and $R_B \delta(\eta)=0$ in the last step; the first of these statements follows as $\im(d^{p-1})=W_p\subset W_j$ and $\delta(\eta)=\delta(A(\beta))$, and the second follows from again $\delta(\eta)=\delta(A(\beta))$ and $R_B\delta(A(\beta))=0$, using the assumptions. This completes the proof.
\end{proof}

In order to find $\beta$ and $\ell$ as in \Cref{lem:correction}, we will need some machinery.

\begin{definition}
For $(X,x),(Y,y)\in A\lmod$, we define a $\kk$-linear mapping
$$ \cK=\cK_{X,Y}\: \Hom_\kk(X,Y)\to\Hom_A(A(X),A(Y)), \quad
f\mapsto A([-g^{\{y\}}(T) f g^{\{x\}}(-T)]_{T^{p^{r-1}}}) .
$$
\end{definition}

In other words, $\cK$ is the negative coefficient of $T^{p^{r-1}}$ in the natural action of the additive group on the internal hom-space $X^*\o Y$, transported along $A(\cdot)$, and $R_X$ is one of the operators arising in \Cref{lem:tensor-product}. Note that $R_X^2=0$ by \Cref{lem:involution}.

\begin{definition}
For any $(X,x)\in A\lmod$, a \emph{weight filtration} is a descending vector space filtration
$$
X=\cF_0 X\supset \cF_1 X\supset\dots 
$$
such that $\cF_{p^{r-1}}X=0$ and $x^{[a]} (\cF_i X) \subset \cF_{i+a} X$ for all $a,i\ge0$.
\end{definition}

\begin{lemma} \label{lem:C-K} \begin{enumerate}
\item[(a)] For any composition $bfa$ of $\bk$-linear maps $a,b,f$ between $A$-modules, where $a,b$ are $A$-linear, 
$$
\cK(bfa) = A(b) \cK(f) A(a) . 
$$

\item[(b)] For any $\bk$-linear map $f\:X\to Y$ between $A$-modules,
$$
R_Y \cK(f) - \cK(f) R_X = R_Y A(f) R_X .
$$

\item[(c)] For any $\bk$-linear map $f\:X\to Y$ between $A$-modules with weight filtrations, 
$$
\cK(f) (A(\cF_i X)) \subset A(\cF_{i+1}Y) ,
$$
and if $f$ is a filtered linear map, then $\cK(f)=0$.

\item[(d)] For any $A$-linear map $f\:X\to Y$ between $A$-modules,
$$
R_Y A(f) = A(f) R_X .
$$
\end{enumerate}
\end{lemma}

\begin{proof} (a) follows directly from the definition of $\cK$.

(b) We set $q:=p^{r-1}$,
$$
x' := (-\id_A\o x_i)_i ,\quad
y' := (-\id_A\o y_i)_i ,\quad
z'_X := (z_i\o \id_X)_i ,\quad z'_Y := (z_i\o \id_Y)_i ,
$$
where $x=(x_i)_i$ and $y=(y_i)_i$ are the $(r-1)$-tuples specifying the $A$-action on $X$ and $Y$.  
Then by \Cref{eq:star}, 
$$
g^{\{z'_X\}} g^{\{x'\}} = g^{\{z'_X*x'\}} (1-T^q R_X) ,\quad
g^{\{y'\}} g^{\{z'_Y\}} = g^{\{z'_Y*y'\}} (1-T^q R_Y) ,
$$
hence,
$$
g^{\{z_Y'*y'\}}(-T) A(f) g^{\{z'_X*x'\}}(T) = (1-T^q R_Y) \underbrace{ g^{\{y'\}}(-T) A(f) g^{\{x'\}}(T) }_{=:K(T)} (1+T^q R_X) .
$$
Taking coefficients of $T^{2q}$ and using $[K]_{T^q}=-\cK(f)$ produces the desired formula.

(c) Up to scalar extension $A(\cdot)$, $\cK(f)$ is a linear combination of terms of the form $y^{[a]} f x^{[q-a]}$, for $1\le a\le q-1$. Such a term vanishes on $\cF_i X$ for $a\le i$, as $\cF_q=0$. The remaining terms send $\cF_i X$ to $\cF_{i+1} Y$. If $f$ is filtered, then all terms send $\cF_i X$ to $\cF_q Y=0$.

(d) With $x,y$ as above, $A$-linearity means
$$
f g^{\{x\}}(T) = g^{\{y\}}(T) f .
$$
Evaluating at $-T$, tensoring with $A$, that is with $g^{\{z\}}$, and taking the coefficient of $T^{q}$ yields the assertion.
\end{proof}

Set
$$
\mfrak := (z_0,\dots,z_{r-2}) \quad \subset A ,
$$
the maximal ideal of codimension $1$. Hence, for any $f\:A(X)\to A(Y)$ in $A\lmod$, we can consider the reduction modulo $\mfrak$, a linear map $\ol{f}\:X\to Y$. Note that $\ol{R_X}=0$, as $\deg(g^{\{x\}})<p^{r-1}$. 

We say that a map $f\:X\to Y$ between filtered spaces $X,Y$ with descending filtrations \emph{strictly raises the filtration} if $f(\cF_{i}X)\subset\cF_{i+1}Y$ for all $i$. 

\begin{lemma} \label{lem:extension} Consider $A$-modules with weight filtrations $X,Y,B$, and a filtered $A$-linear map $\partial\:X\to Y$. Let $f\:X\to B$ be a filtered linear map whose restriction to $\ker(\partial)$ is $A$-linear. Then there is an $A$-linear map $\ell\:A(Y)\to A(B)$ satisfying
$$
\ell R_Y=R_B\ell,\quad
\ell A(\partial)=A(f)R_X-R_B A(f),
$$
such that the reduction $\ol\ell\:Y\to B$ factors through $Y/\im(\partial)$ and strictly raises the filtration.
\end{lemma}

\begin{proof} We set 
\[h:=A(f) R_X-R_B A(f):\; A(X)\to A(B).\] By our assumptions, $f$ restricted to $\ker(\partial)$ is $A$-linear, so $h(\ker A(\partial))=0$ by \Cref{lem:C-K}(d). As $f$ is assumed filtered, using \Cref{lem:C-K}(c) and (b), we get $\cK(f)=0$ and hence $R_B A(f) R_X=0$. This implies
\begin{equation}
R_B h = 0 = h R_X .  \label{eq:hom1}    
\end{equation}

We claim that we can choose a linear (split) map $t\:Y\to X$ such that $\partial t \partial=\partial$ and $\partial t$ and $t\partial$ are both filtered linear maps. Indeed, pick filtered complements $X'$ to $\ker\partial\subset X$ and $Y'$ to $\im\partial\subset Y$, then $t$ is defined by
$$
t|_{Y'} = 0 , \quad
t|_{\im\partial} = (\partial|_{X'}^{\im\partial})^{-1} ,
$$
where the upper and lower indicate (co-)restriction. This shows the claim.

Set $\ell_0:=h A(t)$. Then
$$
\ell_0 A(\partial) =  h A(t\partial) = h - h A(1-t\partial) = h,
$$
as $(1-t\partial)(X)\subset\ker\partial$.

Using \Cref{lem:C-K}(a,c), we get that the following compositions are all zero:
$$
A(\partial) \cK(t) ,\quad \cK(t) A(\partial) ,\quad
\cK(ft) A(\partial) .
$$
In particular, $\im\cK(t)\subset\ker A(\partial)$, so by our assumption
\begin{equation}
A(f) R_X \cK(t) = R_B A(f) \cK(t) . \label{eq:hom2}    
\end{equation}

Set $\ell:=\ell_0+ A(f)\cK(t)-\cK(ft)$. It is an $A$-linear map by construction and $\ell A(\partial)=h$. We claim that $\delta(\ell)=0$, for $\delta(H):=H R_Y-R_B H$. Indeed,
$$
\delta(\ell_0) 
=  A(f) R_X A(t) R_Y - R_B A(ft) R_Y
= \delta( -A(f) \cK(t) + \cK(ft) ) ,
$$
using \Cref{eq:hom1} in the first step, and using \Cref{eq:hom2} and \Cref{lem:C-K}(b) in the second step.

Finally, $\ol{\ell_0}=0$ as $\ol{R_X}$ and $\ol{R_B}$ are zero. The remaining terms of $\ell$ annihilate $\im(A(\partial))$, as noted before, and strictly raise the filtration by \Cref{lem:C-K}, as $f$ is filtered. Hence, their reductions induce maps on $Y/\im(\partial)$ that strictly raise the filtration.
\end{proof}

Having set up the machinery, we can now show a lifting result on which the desired Steinberg splitting will rest. 

For any $A_r$-module $X$, we can reduce the $A_{r+1}$-module $\cT(X)=A\o X$ modulo $\mfrak=(z_0,\dots,z_{r-2})$. As the $A$-action on $\cT(X)$ is on the first factor, this reduction can be identified with $X$ as $\kk[z_{r-1}]/(z_{r-1}^p)$-modules. Note that, as $\ol{R_X}=0$, there is no other non-trivial action on the reduction.

\begin{proposition} \label{prop:lifting} For each $Y\in A_r\lmod$ and $1\le j\le p$, there are weight filtrations of $Y$ and of $B_j:=B_j(Y)$ (as $A=A_{r-1}$-modules) together with an $\ol A_{r+1}$-morphism $F\:\cT(Y)\to\cT(B_j\boxtimes J_j)$, such that the action of $B_j(-)$ on the reduction $Y\to B_j\boxtimes J_j$ modulo $\mfrak$ is a filtered map
$$
B_j = B_j(Y)\to B_j(B_j\boxtimes J_j) \cong B_j
$$
whose associated graded is the identity. 
\end{proposition}

\begin{proof}
Set 
$$
W_j:=W_j(Y) ,\quad
y_i:=\begin{cases} z_i|_Y & i<r-1 \\ 0 & i \ge r-1 \end{cases},\quad
d:= z_{r-1}|_Y \quad 
\tforall j\ge1, i\ge0.
$$
Define $\cF_i Y := \sum_{a\ge i} y^{[a]} Y$ for all $i\ge0$. We can realise $B_j$ as a subspace of $Y/\im(d^j)$, since $\im(d^j)\cap W_j=W_{j+1}$. Let $\cF_i B_j$ denote the induced filtration on $B_j$. These are weight filtrations, and $d:Y\to Y$ is filtered.

Pick a filtered linear retraction of $B_j\hookrightarrow Y/\im(d^j)$. Pre-composing it with the quotient map yields a filtered linear map $\beta\:Y\to B_j$ such that $\beta d^j=0$ and $\beta|_{W_j}$ is the natural quotient map. Using \Cref{lem:C-K}(b,c) we hence see that $\beta$ is as assumed in \Cref{lem:correction}, i.e., satisfies \Cref{eq:correction-beta}.

To apply \Cref{lem:extension} with 
$$
\partial := d^j, \quad
f := \beta j d^{j-1},\quad
B := B_j ,
$$
we note that, indeed, $\partial$ and $f$ are filtered, $\partial$ is $A$-linear, and the restriction $f|_{\ker(\partial)}=f|_{\ker(d^j)}$ is $A$-linear, as $d^{j-1}(\ker(d^j))=W_j$, and $\beta|_{W_j}$ is the natural quotient map.
Hence, from \Cref{lem:extension} and \Cref{lem:correction}, we get an $A$-linear map $\ell\: A(Y)\to A(B)$ such that
$$
\eta:= A(\beta)-\ell \quad
\in \Hom_A(A(Y),A(B)) \cong \Hom_A(\cT(Y),A(B))
$$ lies in $H(Y,B,j)$ as defined in \Cref{lem:jordan} and whose reduction $\ol{\ell}$ modulo $\mfrak$ factors via $Y/\im(d^j)$ and strictly raises the filtration. For $j=p$, we can take $\ell=0$.

Hence, from \Cref{lem:jordan}, we get an $\ol A_{r+1}$-linear $F\:\cT(Y)\to\cT(B\boxtimes J_j)$ such that $\pi F=\eta$, where $\pi\:A\o B\o J_j\to A\o B$ projects to the coefficient of $v_j\in J_j$.

It remains to see that the reduction $\ol\eta$ of $\eta$ modulo $\mfrak$, and hence also that of $F$ by \Cref{lem:jordan}, induces the identity map on $\gr(B_j)$. Write $\ol\ell=Y\xrightarrow{\ell_1} Y/\im(d^ j)\xrightarrow{\ell_2} B$, where $\ell_2$ raises the filtration. As $\ol\eta=\beta-\ol\ell$, we see that $\ol\eta|_{W_j}$ induces the map $\id_B - \ell_2|_B$ on $B\subset Y/\im(d^j)$, which indeed has the identity as its associated graded. Modulo $\mfrak$, the maps $F A(d)$ and $A(z_{r-1}) F$ agree, and $\ol F$ sends $W_j\subset\ker(d)$ into $B\o\kk v_j$, with coefficient $\ol\eta|_{W_j}$. Under the natural identification $B_j(B\boxtimes J_j)\cong B$, it therefore induces the identity on $\gr(B_j)$.
\end{proof}

Combining \Cref{prop:lifting} for the various values of $j$ yields a first splitting result.

\begin{proposition} \label{prop:splitting} For all $Y\in A_r\lmod$, 
$$
\cT(Y) \cong \bigoplus_{1\le j\le p} \cT(B_j(Y) \boxtimes J_j) .
$$
In particular, in $\Pol$, the tensor product $A\o Y$ decomposes as a direct sum of tensor products $A\o Y_j$, where $Y_j$ corresponds to $B_j(Y)\boxtimes J_j$.
\end{proposition}

\begin{proof} From \Cref{prop:lifting} applied to all $1\le j\le p$, we get an $\ol A_{r+1}$-linear morphism
$$
F\: \cT(Y) \to \bigoplus_{1\le j\le p} \cT(B_j(Y)\boxtimes J_j),
$$
whose reduction $\ol F$ modulo $\mfrak$ induces the identity on all $\gr(B_j(Y))$, with the appropriate corestriction on the target. Each induced map on $B_j$ is an isomorphism, since its associated graded is the identity. The filtrations by $W_j$ then show that $\ol F$ restricts to an isomorphism between the kernels of the respective $z_{r-1}$-actions. Every nonzero $z_{r-1}$-stable subspace intersects non-trivially with $\ker(z_{r-1})$, so $\ol F$ is injective. A dimension count shows that $\ol F$ is an isomorphism. Hence, by Nakayama's lemma, $F$ is an isomorphism. The implication for $\Pol$ follows from \Cref{lem:equivalences}.
\end{proof}

We can now apply \Cref{prop:splitting} recursively.

\begin{definition} For any polynomial representation $(Y,g)\in\Pol_{r-1}$ and $k\ge0$, let $(Y^{(k)},g^{(k)})\in\Pol_{r-1+k}$ be the \emph{$k$-th Frobenius twist}, that is, the polynomial representation with $Y^{(k)}=Y$ and $g^{(k)}(T)=g(T^{p^k})$.
\end{definition}

Let $J_j$ be the polynomial representation corresponding to the $A_1$-module $(\kk^j, d)$, where $d$ is the nilpotent Jordan block of size $j$. Note that this is consistent with our use of the notation $-\boxtimes J_j$ so far.

\begin{lemma} \label{lem:frob-twist} For $Y\in A_{r-1}\lmod$ and $1\le j\le p$, $Y\boxtimes J_j\cong Y\o J_j^{(r-1)}$ as polynomial representations.
\end{lemma}

\begin{proof} This follows by a direct computation from the definitions.
\end{proof}

For $0\le a< p^r$ with $p$-adic expansion $a=a_0+pa_1+\dots+p^{r-1} a_{r-1}$, set
$$
J(a) :=  J_{a_0+1}\boxtimes\dots\boxtimes J_{a_{r-1}+1}
 = J_{a_0+1}^{(0)}\o\dots\o J_{a_{r-1}+1}^{(r-1)} \quad
 \in A_r\lmod\simeq\Pol_r .
$$
Here, by definition, $z_i\in A_r$ acts only on the factor $J_{a_i+1}$. 
We also set
$$
H_X(a) := B_{a_0+1}(\dots B_{a_{r-1}+1}(X)\dots)
\quad\in A_0\lmod\simeq\Vect
\quad\tforall X\in A_r\lmod\simeq\Pol_r .
$$

\begin{theorem} \label{thm:St-poly} For all $X\in\Pol_r$, 
$$
J(p^{r-1}-1) \o X \cong \bigoplus_{0\le a< p^r} (J(p^{r-1}-1)\o J(a))^{\oplus\dim H_X(a)} .
$$
\end{theorem}

\begin{proof} This follows immediately from \Cref{prop:splitting} using \Cref{lem:frob-twist} by an induction in $r$, noting that $A_{r-1}\cong J(p^{r-1}-1)$ and, if $r\ge2$, $J(p^{r-1}-1)\cong J(p^{r-2}-1)\o J(p-1)^{(r-2)}$.
\end{proof}

\subsection{Steinberg splitting for elementary abelian groups} \label{sec:St-Cpr}
The above results can be translated directly into results on representations of elementary abelian $p$-groups.

We continue with a fixed $r\ge1$. Let $C_p^r$ be the elementary abelian $p$-group of order $p^r$. Fix an embedding $\iota\:C_p^r\hookrightarrow(\kk,+)$. The results discussed will turn out to hold independently of this choice. For each $i\ge0$, let $L(i)$ be the restriction to $C_p^r$ of the simple $\SL_2=\SL_2(\kk)$-module with highest weight $i$ along the inclusion
$$
C_p^r \xrightarrow{\iota} (\kk,+) \xrightarrow{\alpha\mapsto \left(\begin{smallmatrix} 1&0\\\alpha&1 \end{smallmatrix}\right)} \SL_2 .
$$

We define a linear map 
$$
\phi\: \kk C_p^r \to A_r , \quad
e \mapsto g^{\{z\}}(\iota(e)) = \prod_{0\le k\le r-1} \exp(\iota(e)^{p^k} z_k) = \sum_{0\le a<p^r} \iota(e)^a z^{[a]}  \quad
\tforall e\in C_p^r .
$$

\begin{lemma} \label{lem:Er-Ar} 

(a) $\phi$ is an algebra isomorphism. 

(b) The monoidal restriction functor 
$$
\Res\:\Pol\simeq\Rep\mG_a\to \Rep C_p^r ,\quad
(V,g) \mapsto (V, e\mapsto g(\iota(e))) ,
$$
restricts on $\Pol_r$ to the equivalence of abelian categories
$$
\Pol_r\xrightarrow{\Psi_r} A_r\lmod \xrightarrow{\phi^*} \Rep C_p^r. 
$$

(c) For all $0\le a<p^r$, $\Res(J(a))\cong L(a)$.
\end{lemma}

\begin{proof}
(a) Additivity of $g^{\{z\}}(T)$ in $T$ implies that $\phi$ is an algebra map. Using bases $\{z^{[a]}\}_{0\le a<p^r}$ of $A_r$ and $\{e\}_{e\in C_p^r}$, $\phi$ is given by the matrix with entries $\iota(e)^a$. The matrix is invertible, as the $p^r$ values $\iota(e)$ are all distinct.

(b) Restriction is clearly monoidal, a direct computation shows the two functors agree.

(c) follows from the Steinberg tensor product theorem for $\SL_2$, since for $0\le b<p$, the restriction to $\mG_a$ of the simple $\SL_2$-representation with highest weight $b$ is $J_{b+1}$.
\end{proof}

For any $p$-nilpotent operator $d$ of a vector space and $1\le j\le p$, we define generalised cohomologies 
$$
B_j(d) := \frac{\ker(d)\cap\im(d^{j-1})}{\ker(d)\cap\im(d^j)} .
$$
Every operator commuting with $d$ acts on these subquotients, while $d$ itself acts as zero. Note that for any $A_{r+1}$-module $Y$ and any $j$, $B_j(z_r|_Y)$ in this sense is $B_j(Y)$ in the sense of \Cref{def:B}.

We set $c_i:=\phi^{-1}(z_i)$ in $\kk C_p^r$ for $0\le i\le r-1$. 

\begin{definition}
For any $X\in\Rep C_p^r$ and any $0\le a< p^r$ with $p$-adic expansion $a=a_0+pa_1+\dots+p^{r-1} a_{r-1}$, set
$$
H_X(a) := B_{a_0+1}(c_0| \dots B_{a_{r-1}+1}(c_{r-1}|X)\dots) 
$$
(where $c_i|$ denotes the action of $c_i$ on what follows).
\end{definition}

\begin{theorem} \label{thm:wet-spaghetti} For any $X\in\Rep C_p^r$, and $St_{r-1}:=L({p^{r-1}-1})$,
$$
St_{r-1}\o X \cong \bigoplus_{0\le a< p^r} (St_{r-1}\o L(a))^{\oplus\dim H_X(a)} .
$$
In particular, the left-hand side is the restriction of a tilting module.
\end{theorem}

\begin{proof} The isomorphism follows directly from \Cref{thm:St-poly} using \Cref{lem:Er-Ar}. Note that in particular, $\Res(J(p^{r-1}-1))\cong L(p^{r-1}-1)=St_{r-1}$.

The fact that the right-hand side, and hence also the left-hand side, is the restriction of a tilting module follows directly from \cite[Lemma~3.1.6]{CF2}, which shows that $St_{r-1}\o L(a)$ is a restricted tilting module for $0\le a<p^r$.
\end{proof}


\section{The Benson--Etingof--Ostrik conjecture} \label{sec:BEO}

\subsection{V-functors lead to fibre functors}

For this subsection, we fix a $\bV$-functor $\Theta\:\Rep S_{p^r}\to \cV$ to a finite tensor category $\cV$, for some $r\in\mZ_{>0}$. We also assume that there exists an additive symmetric monoidal equivalence $T:\cV\to\cV$ with the property that $T(\lambda f)=\lambda^{p^r}T(f)$ for $\lambda\in\bk$ and all morphisms $f$ in $\cV$. For a tensor category $\cC$ we denote the resulting $\bO$-functor, see \S\ref{sec:VOfun}, as
\[\Phi_{\cC}:\cC\to \cC\boxtimes\cV.\]

\begin{theorem}\label{thm:general}
    Let $\cC$ be a finite tensor category that is $\Phi$-exact. 
    \begin{enumerate}
        \item If $p>2$, then there exists a tensor functor
    $\cC\to\Ver_p\boxtimes\cV.$
    \item If $p=2$,  then there exists a tensor functor
    $\cC\to\Ver_4^+\boxtimes\cV.$
    \end{enumerate}
\end{theorem}

The rest of the subsection is devoted to the proof. We always let $\cC$ be a finite tensor category. Recall from \cite[\S 7]{EOf} that there is a tensor subcategory $\cC_{\ex}\subset \cC$ which is the full subcategory of objects $X$ with $\FPdim(\Fr(X))=\FPdim(X)$, while all other objects satisfy $\FPdim(\Fr(X))<\FPdim(X)$. We will use the following observation, for finite tensor categories $\cC,\cD$:
\begin{equation}
    \label{eq:exex}
    \mbox{Any tensor functor $\cC\to\cD$ restricts to $\cC_{\ex}\to\cD_{\ex}$.}
\end{equation}
This follows immediately since tensor functors preserve Frobenius-Perron dimension, see \cite[Proposition~3.3.13]{EGNO}, and $\Fr$ commutes with tensor functors. The observation \eqref{eq:exex} can actually be generalised beyond finite tensor categories, by \cite[Theorem~8.2.1]{CEO2}.

\begin{lemma}\label{lem:VW} Consider $s\in\mZ_{>0}$ and simple objects $V,W\in\cC$ with $V,W\not\in\cC_{\ex}$. If
\[[V^{\otimes p^s}:W]\not\equiv 0\mod p,\]
then $\FPdim(W)<\FPdim(V)$.

\end{lemma}
\begin{proof}
    Consider the $p$-th tensor power as a functor $\cC\to\cC\boxtimes \Rep C_p$. Consider arbitrary simple $X,Y\in\cC$. It follows for instance from \cite[\S 2.2]{BensonBook} that $X^{\otimes p}$ has a filtration whose subquotients are given by $i$ times $F_i(X)$, for $1\le i\le p$, where $F_i$ for $i<p$ is as in \eqref{eq:classicFr}. In particular
    \[[X^{\otimes p}:Y]\;\equiv\; \sum_{i=1}^{p-1}i[F_i(X):Y]\mod p.\]
    By assumption $\FPdim(\Fr(V))< \FPdim(V)$, so certainly $\FPdim(F_i(V))<\FPdim (V)$. This proves a (stronger version of) the case $s=1$. It also proves that if $X\in\cC_{\ex}$ then $[X^{\otimes p}:Y]\not=0$ modulo $p$ forces $Y\in\cC_{\ex}$ by \eqref{eq:exex} and the fact that $\cC_{\ex}$ is Frobenius exact, see \cite[Theorem~6.4]{EOf}.

    Now assume that we have proved the lemma for some $s>0$ and consider $s+1$. For $V,W$ as in the lemma, we have
    \[[V^{\otimes p^{s+1}}: W]\;\equiv\; \sum_{L} [L^{\otimes p}:W][V^{\otimes p^s}:L]\mod p,\]
where $L$ runs over all simples in $\cC$, by vanishing of mixed terms modulo $p$. If $L\in\cC_{\ex}$, then $[L^{\otimes p}:W]$ vanishes modulo $p$ by last sentence of the previous paragraph. When $L\not\in \cC_{\ex}$ it follows that $[V^{\otimes p^s}:L]$ being non-zero modulo $p$ forces $\FPdim(L)<\FPdim(V)$ by `the case $s$' and $[L^{\otimes p}:W]$ being non-zero modulo $p$ requires $\FPdim(W)<\FPdim(L)$ by the base case $s=1$ of the lemma. This allows us to conclude the proof by induction on $s$. \end{proof}

\begin{lemma}\label{lem:foriter}
    Assume that $\cC$ is $\Phi$-exact, then \textbf{at least one} of the following is true:
    \begin{enumerate}
        \item $\cC_{\ex}$ contains all simple objects of $\cC$. 
        \item $\Phi_{\cC}$ takes values in $\cC'\boxtimes\cV$ for a proper tensor subcategory $\cC'\subset\cC$.
    \end{enumerate}
\end{lemma}
\begin{proof}
    Let us assume that $\cC_{\ex}$ does not contain all simples. Then we can take a simple object $V_{\max}\not\in\cC_{\ex}$ that has maximal Frobenius-Perron dimension among all such simples. Consider an arbitrary simple $L\in\cC$. We \emph{claim} that no simple constituent of $\Phi_{\cC}(L)$ is of the form $V_{\max}\boxtimes S$ for a simple $S\in\cV$.
    We split the proof of this claim into two cases.

    Assume first that $L\not\in\cC_{\ex}$. Since the Frobenius-Perron dimension of $\Phi_{\cC}(L)$ equals that of $L$, the only hypothetical case to consider is $\FPdim(L)=\FPdim(V_{\max})$ and $\Phi_{\cC}(L)=V_{\max}\boxtimes S$, for $S$ an invertible object. The interpretation of $\Phi_{\cC}$ as in \eqref{eq:reinterpretPhi} yields $\Theta(\Hom(P,L^{\otimes p^r}))\cong S$, with $P$ the projective cover of $V_{\max}$. By preservation of categorical dimension by the linear symmetric monoidal functor $\Theta$, it follows that 
    \[\dim \Hom(P,L^{\otimes p^r})=[L^{\otimes p^r}:V_{\max}]\equiv \pm 1\mod p.\] This leads to a contradiction with \Cref{lem:VW}.

    Secondly, assume that $L\in\cC_{\ex}$. If $V_{\max}\boxtimes S$ is a simple constituent of $\Phi_{\cC}(L)$, then we must have $V_{\max}\boxtimes S\in (\cC\boxtimes \cV)_{\ex}$ by \eqref{eq:exex}. This is contradicted by the following observation, where we abbreviate $V=V_{\max}$
    \[\FPdim(\Fr(V\boxtimes S))=\FPdim(\Fr(V))\FPdim(\Fr(S))<\FPdim(V)\FPdim(S)=\FPdim(V\boxtimes S).\]
    This concludes the proof of the claim.

Denote by $\cC'\subset\cC$ the Serre subcategory generated by all simple $V\in\cC$ for which there exists simple $S\in\cV$ such that $V\boxtimes S$ is a constituent of an object in the essential image of $\Phi_{\cC}$. By construction, this is actually a tensor subcategory, since $V_{\max}\not\in\cC'$ it is a proper tensor subcategory. Finally, $\Phi_{\cC}$ takes values in $\cC'\boxtimes \cV$, since $\cC'\boxtimes \cV$ is a Serre subcategory of $\cC\boxtimes \cV$ containing all simple objects that are constituents of objects in the essential image of $\Phi_{\cC}$.
\end{proof}

\begin{proof}[Proof of \Cref{thm:general}]
For a $\bk$-linear category $\cA$, let $\cA^{(r)}$ denote the same underlying preadditive category, but where the new $\bk$-linear structure satisfies $\lambda\cdot f=\lambda^{p^r}f$, where the right-hand side denotes the original action, for $\lambda\in\bk$ and $f$ a morphism. Since $\bk$ is perfect, the dimension of the endomorphism algebra of $\unit$ remains one when we apply this to a tensor category $\cC$, so that $\cC^{(r)}$ is again a tensor category. Moreover, by our assumed equivalence $T:\cV\to\cV$, we have an equivalence $\cV\cong \cV^{(r)}$ of tensor categories. Consequently, we can and will reinterpret $\Phi_{\cC}$ as a ($\bk$-linear) symmetric monoidal functor
\[\Phi_{\cC}:\cC\to\cC^{(r)}\boxtimes \cV.\]

Assume first that $p>2$. Then $\cC_{\ex}\subset\cC$ is a Serre subcategory by \cite[Corollary~9.5]{EOf} and admits a tensor functor to $\Ver_p$ by \cite[Theorem~8.1]{EOf}.
By  \Cref{lem:foriter}, for some $N>0$ (for example bounded by the number of simples in $\cC$), the functor
\[(\Phi_{\cC^{(r(N-1))}}\boxtimes \cV^{\boxtimes (N-1)})\circ\cdots\circ(\Phi_{\cC^{(r)}}\boxtimes \cV)\circ\Phi_{\cC}:\cC\to\cC^{(rN)}\boxtimes \cV^{\boxtimes N}\]
actually takes values in $\cC_{\ex}^{(rN)}\boxtimes \cV^{\boxtimes N}$.
Using the product tensor functor $(\cV)^{\boxtimes N}\to\cV$ and the above $\cC_{\ex}^{(rN)}\to \Ver_p$, we obtain the desired tensor functor.

If $p=2$, we need some minor adjustments. Let $\widetilde{\cC}_{\ex}$ be the Serre closure of $\cC_{\ex}$. By \cite[Corollary~9.8]{EOf}, $\widetilde{\cC}_{\ex}$ admits a tensor functor to $\Ver_4^+$. We can now repeat the above argument for $p>2$, but with replacements $\cC_{\ex}\leadsto \widetilde{\cC}_{\ex}$ and $\Ver_p\leadsto \Ver_4^+$.
\end{proof}

\begin{remark} \Cref{thm:general}(1) actually remains true for pretannakian categories with finitely many simple objects up to isomorphism (but not necessarily enough projectives). The proof only requires very minor changes, for instance \cite[Theorem~8.1]{EOf} needs to be replaced with the more general \cite[Theorem~1.1]{CEO}. Similarly extending \Cref{thm:general}(2) seems slightly more subtle.
\end{remark}

\subsection{Higher Frobenius functors}\label{sec:useWS}
In this subsection, we fix $r\in\mZ_{>0}$ and $\ulambda\in\mP^{r-1}_f$. As in \S\ref{sec:inSL2} this leads to an embedding $\iota_{\ulambda}:C_p^r\hookrightarrow \SL_2(\bk)$ and $\blambda\in\mP_f^{r-1}$. We thus have a symmetric monoidal restriction functor from $\Tilt \SL_2$ to $\Rep C_p^r$ (by abuse of notation we use the same symbol for a tilting module of $\SL_2$ and its restriction under $\iota_{\ulambda}$). We denote its essential image by $\cD_{\ulambda}$. It was proved in \cite[Theorem~5.3.5]{CF1} that the restriction functor induces an equivalence
\[\Tilt \SL_2/I_r\;\to\;\cD_{\ulambda}/\cK_{\ulambda} ,\]
where $I_r$ is the kernel of the defining functor $\Tilt \SL_2\to\Ver_{p^r}$ and $\cK_{\ulambda}$ is some tensor ideal in $\cD_{\ulambda}$. We invert this symmetric monoidal equivalence in the proposition below.

\begin{proposition}\label{prop:Theta}
    The $\bk$-linear symmetric monoidal functor
    \[\cD_{\ulambda}\to\cD_{\ulambda}/\cK_{\ulambda}\xrightarrow{\sim}\Tilt \SL_2/I_r\hookrightarrow \Ver_{p^r}\]
    has a unique $\bk$-linear symmetric monoidal extension $\Theta^{\ulambda}:\Rep C_p^r\to\Ver_{p^r}$. For $C_p^r$ interpreted as the permutation group $E_r<S_{p^r}$, this is a $\bV$-functor.
\end{proposition}
    \begin{proof}
Uniqueness and existence of the extension along $\cD_{\ulambda}\subset \Rep C_p^r$ follow from \Cref{thm:wet-spaghetti} and \cite[Lemma~5.4.6]{CF1} (for the morphism $f$ in the latter lemma, we can take the evaluation $St_{r-1}^{\otimes 2}\to\unit$). That the resulting $\Theta^{\ulambda}$ is actually a $\bV$-functor was proved (assuming the validity of \Cref{conj:CF}, which is established in \Cref{thm:wet-spaghetti}) in  the proof of \cite[Theorem~5.4.3]{CF1}.  
    \end{proof}

\Cref{prop:Theta} thus implies that we obtain an $\bO$-functor $\Phi^{\ulambda}$, as predicted in \cite[Conjecture~5.4.2(3)]{CF1}. These are generalisations of the Frobenius functor $\Fr$, which we recover for $r=1$ (and the unique point in $\mP^0$). To detect $\Phi^{\ulambda}$-exactness below, we add the following lemma.

\begin{lemma}\label{lem:kerTheta}
    The following conditions are equivalent on $M\in \Rep C_p^r$:
    \begin{enumerate}
        \item $St_{r-1}\otimes M$ is not projective;
        \item $\blambda \in \supp(M)$;
        \item $\Theta^{\ulambda}(M)\not=0$.
    \end{enumerate}
\end{lemma}
\begin{proof} 
    By \cite[Corollary~4.2.7]{CF2}, we have $\supp St_{r-1}=\{\blambda\}$, so that equivalence between (1) and (2) follows.

    By construction above, the only indecomposable tilting modules $T(j)$ annihilated by $\Theta^{\ulambda}$ are those with $j\ge p^r-1$. These are precisely the tilting modules that are projective over $C_p^r$, see \cite[Theorem~5.3.5(1)]{CF1}. Since $\Theta^{\ulambda}$ is monoidal, we have
    \[\Theta^{\ulambda}(M)\otimes \Theta^{\ulambda}(St_{r-1})\cong \Theta^{\ulambda}(M\otimes St_{r-1}).\]
    Since $\Theta^{\ulambda}(St_{r-1})$ is not zero and $M\otimes St_{r-1}$ is the restriction of a tilting module by \Cref{thm:wet-spaghetti}, the equivalence between (1) and (3) follows.
\end{proof}

\subsection{Proof of the conjecture}

\begin{theorem}\label{thm:BEO}
    Every finite tensor category admits a tensor functor to $\Ver_{p^n}$ for some $n\in\mZ_{>0}$.
    Every such tensor functor is unique up to isomorphism.
\end{theorem}

We start the proof with some preliminary results. In this section, we let $\cC$ be a finite tensor category.

\begin{lemma}\label{lem:notproj}
    Let $P$ be a projective generator of $\cC$ and let $Q$ be the injective hull of $\unit$ ($Q$ is also projective). Then there exists $s\in\mZ_{>0}$ such that the $E_s$-representation $\Hom(P,Q^{\otimes p^s})$ is non-projective, but its restriction to every proper subgroup of $E_s$ is projective.
\end{lemma}
\begin{proof}
    For every $n\in\mZ_{>0}$, we set
    \[M_n:=\Hom_{\cC}(P,Q^{\otimes n})\;\in\;\Rep S_n.\]
    It has a non-zero $S_n$-invariant vector
    \[P\tto \unit\hookrightarrow Q^{\otimes n},\]
    where the right morphism is the $n$-th tensor power of the defining inclusion $\unit\hookrightarrow Q$ and the left map is an arbitrary epimorphism. If $M_n$ is projective, it thus needs to contain the projective cover $P_0(n)$ (which is also the injective hull) of the trivial representation in $\Rep S_n$ as a direct summand.

By \cite[Proposition~A.1]{CEO}, for every $n\in\mZ_{>0}$ we can choose $n'\ge n$ such that for the canonical inclusion $S_n<S_{n'}$
    \[\bk S_n\hookrightarrow \bk S_{n'}\to\End_{\bk}(P_0(n'))\]
    is injective.
    Assume for a contradiction that all $M_n$ are projective, then this last observation implies that for every $n$ there exists $n'\ge n$ with
    \[\bk S_n\to\End_{\cC}(Q^{\otimes n})\to \End_{\cC}(Q^{\otimes n'})\to \End_{\bk}( \Hom(P,Q^{\otimes n'}))\]
    injective, so clearly the left-most map is injective too. This is contradicted for finite tensor categories (more generally, for tensor categories of moderate growth) by \cite[Corollary~4.11]{CEO}.

    We can thus choose the minimal $m\in\mZ_{\ge p}$ for which $M_m$ is not projective. It follows from minimality that the restriction of $M_m$ to any proper Young subgroup of $S_m$ is projective. Indeed, by \cite[Lemma~3.2.4]{CoNew}, the restriction of $M_m$ to a Young subgroup $S_\kappa<S_m$, for a partition $\kappa$ of length $l$, is a direct sum of $S_\kappa$-representations
    \[\Hom(P_{a_1},Q^{\otimes \kappa_1})\boxtimes \Hom(P_{a_2},Q^{\otimes \kappa_2})\boxtimes \cdots\boxtimes \Hom(P_{a_l},Q^{\otimes \kappa_l}),\]
    where $P_{a_i}$ are indecomposable projectives.

    By Chouinard's Theorem, see for instance \cite[Theorem~1.9.1]{BensonBook}, it follows that the restriction of $M_m$ to some elementary abelian $p$-subgroup of $S_m$ is not projective, while the restriction to every \emph{non-transitive} (meaning contained in a conjugate of a proper Young subgroup)  elementary abelian $p$-subgroup of $S_m$ is projective. The conclusion thus follows from \Cref{lem:ESn}.
\end{proof}

Now we consider the $\bO$-functors $\Phi^{\ulambda}$ resulting from \S\ref{sec:useWS}.

\begin{corollary}\label{cor:exactforsome}
    $\cC$ is $\Phi^{\ulambda}$-exact for some $r\in\mZ_{>0}$ and $\ulambda\in\mP_f^{r-1}$.
\end{corollary}
\begin{proof}
Consider $P,Q,r$ as in \Cref{lem:notproj}.
The latter, combined with \Cref{lem:support} and \Cref{lem:bijP}, implies that there exists $\ulambda\in\mP_f^{r-1}$ such that $\blambda$ is in the support of $\Hom(P,Q^{\otimes p^r})$, viewed as an $E_r$-representation. By \Cref{lem:kerTheta} we find $\Theta^{\ulambda}(\Hom(P,Q^{\otimes p^r}))\not=0$. The conclusion thus follows from \cite[Proposition~5.2.4]{CF1}.
\end{proof}

\begin{proof}[Proof of \Cref{thm:BEO}]
In order to be able to use \Cref{thm:general}, we remark that $\Ver_{p^r}$ has the desired twisted linear autoequivalence $T$, see \cite[\S 4.9]{BEO} or \cite[Remark~2.3.8]{CF1}.

    By the combination of \Cref{cor:exactforsome} and \Cref{thm:general}, there is always some $r\in\mZ_{>0}$ such that there exists a tensor functor
    \[\cC\to\Ver_{p^2}\boxtimes \Ver_{p^r}.\]
    By taking $n=\max(2,r)$ and the tensor product tensor functor $\Ver_{p^2}\boxtimes \Ver_{p^r}\to\Ver_{p^n}$ the first assertion follows.

    Uniqueness of the tensor functor is reduced to \cite[Conjecture~B.6]{EOf}\footnote{Note that the numbering is different in the arxiv version.} in \cite[Remark~6.1.9]{CEO2}. The conjecture was shown in \cite[Theorem~7.1]{CSZ}.
\end{proof}

\begin{remark}\label{rem:toomany}
\Cref{thm:general}  shows that, for some $\iota_{\ulambda}:C_p^r\hookrightarrow (\bk,+)\subset\SL_2(\bk)$, if a finite tensor category $\cC$ is $\Phi^{\ulambda}$-exact then it admits a tensor functor $\cC\to\Ver_{p^r}$. However, for most $\ulambda$, $\Ver_{p^r}$ itself will not be $\Phi^{\ulambda}$-exact. This is already explored for $r=2=p$ in \cite{CF1}. As it happens $\Ver_{p^r}$ will be $\Phi^{\ulambda}$-exact for $\iota_{\ulambda}$ that identifies $C_p^r$ with $\mF_{p^r}\subset \bk$, although we do not prove this in the current paper. Consequently, \Cref{thm:general} can actually be upgraded to the statement that there exists a tensor functor $\cC\to\Ver_{p^r}$ if and only if $\cC$ is $\Phi^{\ulambda}$-exact for such a $\ulambda$. After the proof, it thus becomes clear that most of the $\bO$-functors we consider play no role.
\end{remark}

\subsection{Some consequences}\label{sec:cons}
We present some of the results that follow from the resolution of the finite Benson--Etingof--Ostrik conjectures.

Recall, see \cite[Definition 6.1.13]{EGNO}, that a finite (not necessarily symmetric) tensor category~$\cC$ is integral if $\FPdim(X)\in\mZ$ for all $X\in\cC$. Super tannakian tensor categories and tensor categories that fibre over $\Ver_4^+$ are clearly integral, and the following result shows that these are the only (symmetric) ones.

\begin{proposition}
    Let $\cC$ be an integral finite tensor category.
    \begin{enumerate}
        \item If $p>2$, then there is tensor functor $\cC\to\sVec$.
        \item If $p=2$, then there is tensor functor $\cC\to\Ver_4^+$.
    \end{enumerate}
\end{proposition}
\begin{proof}
    By preservation of Frobenius-Perron dimension and Theorem~\ref{thm:BEO} it suffices to observe that the only objects of integral Frobenius-Perron dimension in $\Ver_{p^\infty}$ belong to the subcategory $\sVec$, if $p>2$, or $\Ver_4^+$ if $p=2$.

    Assume $p>2$. By \cite[Theorem~4.5(iv)]{BEO}, the only $X\in\Ver_{p^n}^+$ with $\FPdim(X)\in\mZ$ are in $\operatorname{Vec}$. Any $Y=A\oplus B\boxtimes\overline{\unit}$ in $\Ver_{p^n}^+\boxtimes \sVec=\Ver_{p^n}$ has the same Frobenius-Perron dimension as $A\oplus B\in \Ver_{p^n}^+$, so $\FPdim(Y)\in\mZ$ forces $A\oplus B\in \operatorname{Vec}$.

    The case $p=2$ follows similarly from \cite[Corollary 2.2]{BE} and the fact that the only objects in $\Ver_{2^\infty}$ that only have trivial simple constituents belong to $\Ver_4^+$, see \cite[Proposition 2.5]{BE}.
\end{proof}

The following result gives an estimate on which higher Verlinde category a finite tensor category fibres over based on the Frobenius--Perron dimension of a projective generator.

\begin{proposition} Let $\cC$ be a finite tensor category with projective generator $P$ such that $\FPdim(P) < p^{r-1}$, for some $r\ge2$. Then $\cC$ admits a tensor functor for $\Ver_{p^r}$.
\end{proposition}

\begin{proof} By \Cref{thm:BEO}, there is a tensor functor $F\:\cC\to\Ver_{p^s}$, for some $s\ge1$. Pick the minimal such $s$. Let $\cD\subset\Ver_{p^s}$ be the minimal tensor category containing the essential image of $F$. Then any object in $\cD$ is a subquotient of an object in the essential image of $F$. In particular, any simple object is a subquotient of $F(P)$, so the number of simple objects in $\cD$ up to isomorphism is at most $\len(F(P))\le\FPdim(P)<p^{r-1}$.

By minimality of $s$ and by the classification of tensor subcategories of the higher Verlinde categories \cite[Corollary~4.61]{BEO}, $\cD$ must be $\Ver_{p^s}$ or $\Ver^+_{p^s}$. In both cases, the number of simple objects up to isomorphism is at least $p^{s-2}$ by \cite[Theorem~4.5]{BEO}. This means $s\le r$.
\end{proof}

The following proposition places significant restrictions on the types of non-abelian $\bk$-linear symmetric rigid monoidal categories that can exist.
\begin{proposition}
   Let $\cA$ be an essentially small $\bk$-linear additive idempotent complete rigid symmetric monoidal category such that $\End(\unit)$ is one-dimensional and all morphism spaces finite-dimensional. Assume that $\cA$ has a unique minimal non-zero tensor ideal such that the set of isomorphism classes of indecomposable objects whose identities lie in the ideal is non-empty and finite. Then there exists a faithful $\bk$-linear symmetric monoidal functor $\cA\to \Ver_{p^n}$ for some $n\in\mZ_{>0}$.
\end{proposition}
\begin{proof}
    By \cite[Theorem~2.4.1(1)]{CEO3}, $\cA$ admits an abelian envelope which, by construction {\it loc. cit.}, is a finite tensor category. The envelope thus admits a tensor functor to $\Ver_{p^n}$ by Theorem~\ref{thm:BEO}.
\end{proof}

We explain how our results resolve other open conjectures or questions.

\begin{proposition}
Assume $p=2$.
\begin{enumerate}
\item \cite[Conjecture~B.5]{CSZ} is true.
    \item Every finite tensor category is geometrically reductive (in the sense of \cite{ComAlg,GR}).  
    \item \cite[Conjecture~9.3.5]{CEO2} is true.
    \item \cite[Conjecture~3.1.6(1)]{ComAlg} is true for finite tensor categories.
\end{enumerate}
If $p>2$, then all these results remain valid provided $\Ver_{p^\infty}$ is maximally nilpotent (meaning \cite[Conjecture~3.2.3]{ComAlg} is true).
\end{proposition}
\begin{proof}
As explained in \cite[Appendix~B]{CSZ}, the conjecture in (1) follows if, for an arbitrary finite tensor category $\cC$, the unique incompressible finite tensor category that $\cC$ admits a surjective tensor functor to (which by Theorem~\ref{thm:BEO} is a tensor subcategory of some $\Ver_{p^n}$) is maximally nilpotent (MN). Since $\Ver_{p^n}$ is conjectured to be MN in \cite[Conjecture~3.2.3]{ComAlg}, and proved to be so for $p=2$ in \cite[Theorem~9.2.1]{CEO2}, the results follow.

Similarly, the other three results follow immediately from \cite[Theorem~4.1.1]{GR} and Theorem~\ref{thm:BEO}.
\end{proof}

The following answers \cite[Question~7.1.4(2,3)]{CEO2} for finite tensor categories. Recall from \cite{CEO2} that the \emph{symmetric growth dimension} $\sgd(X)$ of an object $X$ in a tensor category is defined as
$$
\sgd(X)
:= \limsup_{n\to\infty} \frac{\log(\len(\Sym^{\le n} X))}{\log(n)},
$$
where $\Sym^{\le n}$ is the sum of the symmetric powers up to degree $n$.

\begin{proposition}
Assume $p=2$, let $X$ be an object of a finite tensor category $\cC$, let $F\:\cC\to\Ver_{p^r}$ be a tensor functor, see Theorem~\ref{thm:BEO}. 
Then
$$
\sgd(X)=\sgd(F(X))=\dim\Hom(F(X),\unit)\le\FPdim(X) ,
$$
and $\limsup$ can be replaced by $\lim$ in the definition of $\sgd$. The same is true for all $p$ for which $\Ver_{p^\infty}$ is maximally nilpotent.
\end{proposition}

\begin{proof} As explained in \cite[Remark~7.1.2]{CEO2}, $\len(\cdot)$ can be replaced by $\len(F(\cdot))$ in the definition of $\sgd$. This implies $\sgd(X)=\sgd(F(X))$. The identity $\sgd(F(X))=\dim\Hom(F(X),\unit)$ holds by \cite[Theorem~7.2.2]{CEO2}, as $\Ver_{2^\infty}$ is maximally nilpotent by \cite{CEO2} (or $\Ver_{p^\infty}$ by the alternative assumption). Now $\dim\Hom(F(X),\unit)$ and $\FPdim(X)=\FPdim(F(X))$ can be compared with the multiplicity of the simple object $\unit$ in $F(X)$, which shows the asserted inequality.

Pick an exact sequence
$$
0\to Z\to F(X)\to \unit^h \to 0
$$
in $\Ver_{p^r}$, with $h=\dim\Hom(F(X),\unit)$. Let $\Sym$ denote the symmetric algebra of an object.
Using maximal nilpotence of the target of $F$ again, the image $R$ of $\Sym Z\to\Sym F(X)$ is finite by \cite[Theorem~7.2.2]{CEO2}, and we have an isomorphism $R\o \Sym(\unit^h)\cong\gr(\Sym F(X))$ by \cite[Theorem~3.3.2]{CEO2}, where the associated graded is taken with respect to the filtration $Z\subset F(X)$. Pick $d\ge0$ such that $R_i=0$ for all $i>d$, then for all $n\ge d$ 
$$
\len(\Sym^{\le n} F(X)) = \sum_{0\le i\le d} \len(R_i) \binom{n-i+h}{h} 
$$
is a polynomial in $n$ with positive leading coefficient.

Pick $K\ge1$ greater than the length of the image of any simple object under $F$. It follows that 
$$
\liminf_n \frac{\log(\len(\Sym^{\le n}X))}{\log(n)}
\ge \lim_n \frac{\log(K^{-1} \len(\Sym^{\le n} F(X)))}{\log(n)}
= \limsup_n \frac{\log(K^{-1} \len(\Sym^{\le n} F(X)))}{\log(n)},
$$
which is again $\sgd(F(X))=\sgd(X)$ by \cite[Remark~7.1.2]{CEO2}.
\end{proof}

We conclude the paper by pointing out how our results contribute towards the full conjecture of Benson, Etingof and Ostrik.

\begin{proposition}
    The following statements are equivalent:
    \begin{enumerate}
        \item Every finitely generated tensor category of moderate growth admits a tensor functor to $\Ver_{p^n}$ for some $n\in\mZ_{>0}$.
        \item Every finitely generated tensor category of moderate growth admits a tensor functor to a finite tensor category.
        \item Every incompressible finitely generated tensor category of moderate growth is a finite tensor category.
    \end{enumerate}
\end{proposition}
\begin{proof}
    The implications (1)$\Rightarrow$(2)$\Rightarrow$(3) are obvious and do not require our results. The implication (3)$\Rightarrow$(1) follows immediately from \Cref{thm:BEO} and \cite[Theorem~5.2.1]{CEO2}.
\end{proof}

We point out that for \Cref{conj:BEO} itself to follow as a consequence of the above equivalent properties, one would for instance need to prove that the categories $\Ver_{p^n}$ satisfy the property MN, so that uniqueness of tensor functors, see \cite[Theorem~7.3.1]{CEO2}, would allow one to glue together tensor functors on finitely generated tensor subcategories.

\printbibliography

@misc{Sherman,
author       = {Sherman, Alexander},
  title        = {Semi-exact functors and the conjecture of Coulembier and Flake},
  year         = {2026},
  note         = {In preparation}
}

@book{BensonBook,
    AUTHOR = {Benson, David J.},
     TITLE = {Representations of elementary abelian {$p$}-groups and vector
              bundles},
    SERIES = {Cambridge Tracts in Mathematics},
    VOLUME = {208},
 PUBLISHER = {Cambridge University Press},
   ADDRESS = {Cambridge},
      YEAR = {2017},
      ISBN = {978-1-316-79569-9},
       DOI = {10.1017/9781316795699},
       URL = {https://doi.org/10.1017/9781316795699},
}

@article{BE,
    AUTHOR = {Benson, Dave and Etingof, Pavel},
     TITLE = {Symmetric tensor categories in characteristic 2},
   JOURNAL = {Adv. Math.},
  FJOURNAL = {Advances in Mathematics},
    VOLUME = {351},
      YEAR = {2019},
     PAGES = {967--999},
       DOI = {10.1016/j.aim.2019.05.020},
       URL = {https://doi.org/10.1016/j.aim.2019.05.020},
}

@article{BEO,
    AUTHOR = {Benson, Dave and Etingof, Pavel and Ostrik, Victor},
     TITLE = {New incompressible symmetric tensor categories in positive
              characteristic},
   JOURNAL = {Duke Math. J.},
  FJOURNAL = {Duke Mathematical Journal},
    VOLUME = {172},
      YEAR = {2023},
    NUMBER = {1},
     PAGES = {105--200},
      ISSN = {0012-7094,1547-7398},
   MRCLASS = {18M05},
  MRNUMBER = {4533718},
MRREVIEWER = {Xiao-Kan\ Guo},
       DOI = {10.1215/00127094-2022-0030},
       URL = {https://doi.org/10.1215/00127094-2022-0030},
}

@article{CoNew,
    AUTHOR = {Coulembier, Kevin},
     TITLE = {Inductive systems of the symmetric group, polynomial functors
              and tensor categories},
   JOURNAL = {Ann. Sci. \'{E}c. Norm. Sup\'{e}r. (4)},
  FJOURNAL = {Annales Scientifiques de l'\'{E}cole Normale Sup\'{e}rieure.
              Quatri\`{e}me S\'{e}rie},
    VOLUME = {59},
      YEAR = {2026},
    NUMBER = {4},
     PAGES = {957--1010},
       DOI = {10.24033/asens.2652},
       URL = {https://doi.org/10.24033/asens.2652},
%    EPRINT = {2406.00892},
%ARCHIVEPREFIX = {arXiv},
%PRIMARYCLASS = {math.RT},
}

@article {ComAlg,
    AUTHOR = {Coulembier, Kevin},
     TITLE = {Commutative algebra in tensor categories},
   JOURNAL = {Transform. Groups},
  FJOURNAL = {Transformation Groups},
    VOLUME = {31},
      YEAR = {2026},
    NUMBER = {2},
     PAGES = {1225--1272},
      ISSN = {1083-4362,1531-586X},
   MRCLASS = {18M05 (13B30 14A05 14L17)},
  MRNUMBER = {5057629},
       DOI = {10.1007/s00031-026-09965-z},
       URL = {https://doi.org/10.1007/s00031-026-09965-z},
}

@article{CEO,
    AUTHOR = {Coulembier, Kevin and Etingof, Pavel and Ostrik, Victor},
     TITLE = {On {F}robenius exact symmetric tensor categories},
      NOTE = {With Appendix A by Alexander Kleshchev},
   JOURNAL = {Ann. of Math. (2)},
  FJOURNAL = {Annals of Mathematics. Second Series},
    VOLUME = {197},
      YEAR = {2023},
    NUMBER = {3},
     PAGES = {1235--1279},
      ISSN = {0003-486X,1939-8980},
   MRCLASS = {18M05 (14L15 16T05 20C20)},
  MRNUMBER = {4564264},
MRREVIEWER = {Mee\ Seong\ Im},
       DOI = {10.4007/annals.2023.197.3.5},
       URL = {https://doi.org/10.4007/annals.2023.197.3.5},
}

@article{CEO2,
    AUTHOR = {Coulembier, Kevin and Etingof, Pavel and Ostrik, Victor},
     TITLE = {Incompressible tensor categories},
   JOURNAL = {Adv. Math.},
  FJOURNAL = {Advances in Mathematics},
    VOLUME = {457},
      YEAR = {2024},
     PAGES = {Paper No. 109935, 65},
      ISSN = {0001-8708,1090-2082},
   MRCLASS = {18M05 (14 16G10 18M25 20)},
  MRNUMBER = {4794815},
       DOI = {10.1016/j.aim.2024.109935},
       URL = {https://doi.org/10.1016/j.aim.2024.109935},
}

@article{CF1,
    AUTHOR = {Coulembier, Kevin and Flake, Johannes},
     TITLE = {Towards higher {F}robenius functors for symmetric tensor
              categories},
   JOURNAL = {Selecta Math. (N.S.)},
  FJOURNAL = {Selecta Mathematica. New Series},
    VOLUME = {32},
      YEAR = {2026},
     PAGES = {Paper No. 98},
       DOI = {10.1007/s00029-026-01204-3},
       URL = {https://doi.org/10.1007/s00029-026-01204-3},
%    EPRINT = {2405.19506},
%ARCHIVEPREFIX = {arXiv},
%PRIMARYCLASS = {math.RT},
}

@misc{CF2,
    AUTHOR = {Coulembier, Kevin and Flake, Johannes},
     TITLE = {A conjecture on the tensor ideal for an elementary {$p$}-group
              generated by the restriction of a {S}teinberg module},
      YEAR = {2025},
      %NOTE = {Preprint, arXiv:2507.02722},
    EPRINT = {2507.02722},
ARCHIVEPREFIX = {arXiv},
%PRIMARYCLASS = {math.RT},
       DOI = {10.48550/arXiv.2507.02722},
       URL = {https://arxiv.org/abs/2507.02722},
}

@misc{CEO3,
    AUTHOR = {Coulembier, Kevin and Etingof, Pavel and Ostrik, Victor},
     TITLE = {Tensor ideals of abelian type and quantum groups},
      YEAR = {2025},
      NOTE = {Preprint, arXiv:2511.08859},
    EPRINT = {2511.08859},
ARCHIVEPREFIX = {arXiv},
PRIMARYCLASS = {math.RT},
       URL = {https://arxiv.org/abs/2511.08859},
}

@incollection{Del90,
    AUTHOR = {Deligne, P.},
     TITLE = {Cat\'{e}gories tannakiennes},
 BOOKTITLE = {The {G}rothendieck {F}estschrift, {V}ol. {II}},
    SERIES = {Progr. Math.},
    VOLUME = {87},
     PAGES = {111--195},
 PUBLISHER = {Birkh\"{a}user Boston, Boston, MA},
      YEAR = {1990},
      ISBN = {0-8176-3428-2},
   MRCLASS = {14A99 (12H05 18A99)},
  MRNUMBER = {1106898},
MRREVIEWER = {James\ Milne},
}

@article{Del02,
    AUTHOR = {Deligne, P.},
     TITLE = {Cat\'{e}gories tensorielles},
      NOTE = {Dedicated to Yuri I. Manin on the occasion of his 65th
              birthday},
   JOURNAL = {Mosc. Math. J.},
  FJOURNAL = {Moscow Mathematical Journal},
    VOLUME = {2},
      YEAR = {2002},
    NUMBER = {2},
     PAGES = {227--248},
      ISSN = {1609-3321,1609-4514},
   MRCLASS = {18D10 (18E30 20C99)},
  MRNUMBER = {1944506},
MRREVIEWER = {Viktor\ Ostrik},
       DOI = {10.17323/1609-4514-2002-2-2-227-248},
       URL = {https://doi.org/10.17323/1609-4514-2002-2-2-227-248},
}

@article {tann,
    AUTHOR = {Coulembier, Kevin},
     TITLE = {Tannakian categories in positive characteristic},
   JOURNAL = {Duke Math. J.},
  FJOURNAL = {Duke Mathematical Journal},
    VOLUME = {169},
      YEAR = {2020},
    NUMBER = {16},
     PAGES = {3167--3219},
      ISSN = {0012-7094,1547-7398},
   MRCLASS = {18M15 (14L15 16D90 16T05 20C05)},
  MRNUMBER = {4167087},
MRREVIEWER = {Mee\ Seong\ Im},
       DOI = {10.1215/00127094-2020-0026},
       URL = {https://doi.org/10.1215/00127094-2020-0026},
}

@article {EO-finite,
    AUTHOR = {Etingof, Pavel and Ostrik, Viktor},
     TITLE = {Finite tensor categories},
   JOURNAL = {Mosc. Math. J.},
  FJOURNAL = {Moscow Mathematical Journal},
    VOLUME = {4},
      YEAR = {2004},
    NUMBER = {3},
     PAGES = {627--654, 782--783},
      ISSN = {1609-3321,1609-4514},
   MRCLASS = {18D10 (16W30)},
  MRNUMBER = {2119143},
MRREVIEWER = {Sebastian\ Burciu},
       DOI = {10.17323/1609-4514-2004-4-3-627-654},
       URL = {https://doi.org/10.17323/1609-4514-2004-4-3-627-654},
}

@book{EGNO,
    AUTHOR = {Etingof, Pavel and Gelaki, Shlomo and Nikshych, Dmitri and
              Ostrik, Victor},
     TITLE = {Tensor categories},
    SERIES = {Mathematical Surveys and Monographs},
    VOLUME = {205},
 PUBLISHER = {American Mathematical Society, Providence, RI},
      YEAR = {2015},
     PAGES = {xvi+343},
      ISBN = {978-1-4704-2024-6},
       DOI = {10.1090/surv/205},
       URL = {https://doi.org/10.1090/surv/205},
}

@article{EOf,
    AUTHOR = {Etingof, Pavel and Ostrik, Victor},
     TITLE = {On the {F}robenius functor for symmetric tensor categories in
              positive characteristic},
   JOURNAL = {J. Reine Angew. Math.},
  FJOURNAL = {Journal f\"{u}r die Reine und Angewandte Mathematik},
    VOLUME = {773},
      YEAR = {2021},
     PAGES = {165--198},
       DOI = {10.1515/crelle-2020-0033},
       URL = {https://doi.org/10.1515/crelle-2020-0033},
}

@article{Os,
    AUTHOR = {Ostrik, Victor},
     TITLE = {On symmetric fusion categories in positive characteristic},
   JOURNAL = {Selecta Math. (N.S.)},
  FJOURNAL = {Selecta Mathematica. New Series},
    VOLUME = {26},
      YEAR = {2020},
    NUMBER = {3},
     PAGES = {Paper No. 36, 19},
      ISSN = {1022-1824},
       DOI = {10.1007/s00029-020-00567-5},
       URL = {https://doi.org/10.1007/s00029-020-00567-5},
}

@article{Cou-abelian,
  author = {Coulembier, Kevin},
  title = {Monoidal abelian envelopes},
  journal = {Compos. Math.},
  fjournal = {Compositio Mathematica},
  volume = {157},
  year = {2021},
  number = {7},
  pages = {1584--1609},
  issn = {0010-437X,1570-5846},
  mrclass = {18D15 (18F10 18F20 18M05 20G05)},
  mrnumber = {4277111},
  mrreviewer = {Johannes\ Flake},
  doi = {10.1112/S0010437X21007399},
  url = {https://doi.org/10.1112/S0010437X21007399},
}

@misc{CSZ,
      title={Simple algebras and exact module categories}, 
      author={Kevin Coulembier and Mateusz Stroiński and Tony Zorman},
      year={2025},
      eprint={2501.06629},
      archivePrefix={arXiv},
      primaryClass={math.RT},
      url={https://arxiv.org/abs/2501.06629}, 
}

@misc{GR,
      title={On geometrically reductive tensor categories}, 
      author={Kevin Coulembier},
      year={2026},
      eprint={2605.19167},
      archivePrefix={arXiv},
      primaryClass={math.RT},
      url={https://arxiv.org/abs/2605.19167}, 
}

\end{document}